\documentclass{amsart}

\usepackage[all, cmtip]{xy}
\usepackage{amsmath,amsfonts,mathdots,graphicx,epstopdf}
\usepackage{enumerate}

\theoremstyle{definition}
\newtheorem{definition}{Definition}
\newtheorem{example}{Example}
\newtheorem{remark}{Remark}

\theoremstyle{plain}
\newtheorem{lemma}{Lemma}
\newtheorem{theorem}{Theorem}

\newtheorem{corollary}{Corollary}

\DeclareMathOperator{\res}{res}
\DeclareMathOperator{\End}{End}
\DeclareMathOperator{\im}{im}
\DeclareMathOperator{\Hom}{Hom}
\DeclareMathOperator{\Aff}{Aff}
\DeclareMathOperator{\Ext}{Ext}
\DeclareMathOperator{\rk}{rank}
\DeclareMathOperator{\Gr}{Gr}
\DeclareMathOperator{\dlog}{dlog}
\DeclareMathOperator{\coker}{coker}

\let\coprod=\undefined
\DeclareSymbolFont{cmlargesymbols}{OMX}{cmex}{m}{n}
\DeclareMathSymbol{\coprod}{\mathop}{cmlargesymbols}{"60}

\let\amalg=\undefined
\DeclareSymbolFont{cmsymbols}{OMS}{cmsy}{m}{n}
\DeclareMathSymbol{\amalg}{\mathbin}{cmsymbols}{"71}

\begin{document}

\newcommand{\C}{\mathbb{C}}
\newcommand{\Proj}{\mathbb{P}}
\newcommand{\Z}{\mathbb{Z}}
\newcommand{\calO}{\mathcal{O}}
\newcommand{\calL}{\mathcal{L}}
\newcommand{\calX}{\mathcal{X}}
\newcommand{\calE}{\mathcal{E}}
\newcommand{\calA}{\mathcal{A}}
\newcommand{\calF}{\mathcal{F}}
\newcommand{\calD}{\mathcal{D}}
\newcommand{\calH}{\mathcal{H}}
\newcommand{\calP}{\mathcal{P}}
\newcommand{\CC}{\mathbb{C}}
\newcommand{\QQ}{\mathbb{Q}}
\newcommand{\ZZ}{\mathbb{Z}}
\newcommand{\RR}{\mathbb{R}}
\newcommand{\VV}{\mathbb{V}}
\renewcommand{\AA}{\mathbb{A}}
\newcommand{\HH}{\mathbb{H}}
\newcommand{\PP}{\mathbb{P}}
\newcommand{\shO}{\mathcal{O}}
\newcommand{\shV}{\mathcal{V}}
\newcommand{\shF}{\mathcal{F}}
\newcommand{\shT}{\mathcal{T}}
\newcommand{\ra}{\to}
\newcommand {\llim}  {{\operatorname{lim}}}
\newcommand {\lra}  {\longrightarrow}
\newcommand {\id}  {\operatorname{id}}
\newcommand {\hra} {\hookrightarrow}
\newcommand {\st} {\mathrm{st}}
\newcommand {\aff} {\mathrm{aff}}

\title{An Introduction to Hodge Structures}

\author[S. A. Filippini]{Sara Angela Filippini}
\address{Institut f\"{u}r Mathematik, Universit\"{a}t Z\"{u}rich, Winterthurerstrasse 190, CH-8057 Z\"{u}rich, Switzerland}
\email{saraangela.filippini@math.uzh.ch}

\author[H. Ruddat]{Helge Ruddat}
\address{Math. Institut, Univ. Mainz, Staudingerweg 9, 55099 Mainz, Germany}
 \email{ruddat@uni-mainz.de}

\author[A. Thompson]{Alan Thompson}
\address{Department of Pure Mathematics, University of Waterloo, 200 University Ave W, Waterloo, ON, N2L 3G1, Canada}
\email{am6thomp@uwaterloo.ca}
 \thanks{A. Thompson was supported by a Fields-Ontario-PIMS postdoctoral fellowship with funding provided by NSERC, the Ontario Ministry of Training, Colleges and Universities, and an Alberta Advanced Education and Technology Grant.}

\begin{abstract} We begin by introducing the concept of a Hodge structure and give some of its basic properties, including the Hodge and Lefschetz decompositions. We then define the period map, which relates families of K\"{a}hler manifolds to the families of Hodge structures defined on their cohomology, and discuss its properties. This will lead us to the more general definition of a variation of Hodge structure and the Gauss-Manin connection. We then review the basics about mixed Hodge structures with a view towards degenerations of Hodge structures; including the canonical extension of a vector bundle with connection, Schmid's limiting mixed Hodge structure and Steenbrink's work in the geometric setting. Finally, we give an outlook about Hodge theory in the Gross-Siebert program.
\end{abstract}

\maketitle

\section{Introduction}

At the time of writing this survey, Hodge theory stands as one of the most important and active research areas in algebraic geometry. As such it is a vast subject, with many good introductory surveys and textbooks already available for researchers new to the field. Our aim here is not to compete with this existing literature, nor do we claim that our survey is in any way comprehensive. Instead, in keeping with the thematic goal of this volume, we aim to give a concise introduction to some of the Hodge theoretic ideas that appear in the study of Calabi-Yau varieties, with the intention of providing the reader with the theoretical tools necessary to attack some of the more advanced chapters found herein.

The main part of this survey is divided into three sections. In the first (Sect. \ref{sect:Hodgestructures}), we give basic results about Hodge structures that will be used throughout. The second section (Sect. \ref{sect:vhs}) gives an overview of the theory of variations of Hodge structure, which describes how Hodge structures vary in families. Here we focus heavily on the period map, with the aim of developing the theory needed to study the moduli of K3 surfaces; this will be discussed further in \cite{gmk3s}. Finally, the third section (Sect. \ref{sect:mhs}) gives an introduction to mixed Hodge structures, which we use to study the behaviour of Hodge structures under degeneration. This section aims to provide the reader with the background needed to study the Gross-Siebert approach to mirror symmetry, which will be discussed briefly in the final section of this survey (Sect. \ref{sect:Hodgetoric}) and in significantly more detail in \cite{eagsp}.

For readers interested in pursuing this subject further, as a standard reference for much of the material covered in Sects. \ref{sect:Hodgestructures} and \ref{sect:vhs} we recommend the book by Voisin \cite{htcagI}; we use this book as our main reference in the text wherever possible. When discussing the theory of the period map in Sect. \ref{sect:vhs}, we also find the book by Carlson, M\"{u}ller-Stach and Peters \cite{pmpd} to be an indispensable reference; this book also contains a good discussion of the motivation for introducing mixed Hodge structures. In Sect. \ref{sect:mhs} we switch to the book by Peters and Steenbrink \cite{PS08}, which provides a comprehensive introduction to the theory of mixed Hodge structures. Finally, a general introduction to the Gross-Siebert program may be found in several articles by Gross and Siebert, we recommend  \cite{Gr13} \cite{GS03a} \cite{GS11}.

\section{Hodge Structures} \label{sect:Hodgestructures}

The first section of these notes will introduce the basic notion of a \emph{Hodge structure} and give some fundamental results about them. Historically, Hodge structures were first developed to study compact K\"{a}hler manifolds, but they have since found much broader applications. In our treatment we will reverse the historical order, first giving a brief introduction to the formal aspects of Hodge structures, then specializing to discuss some of the results obtained in the context of compact K\"{a}hler manifolds. This discussion will culminate in the Hodge and Lefschetz decomposition theorems, which are both fundamental results in K\"{a}hler geometry. 

After this, we digress to introduce some machinery from the theory of spectral sequences, which will allow us to give a weakened version of the Hodge decomposition that does not require the K\"{a}hler assumption. Finally, we conclude with a discussion of polarized Hodge structures, which are of vital importance to the study of variations of Hodge structure in Sect. \ref{sect:vhs}.

Throughout this section we will refer frequently to Part II of the book by Voisin \cite[Chaps. 5-8]{htcagI}, which we use as our main reference; complete proofs of all theorems may be found therein.

We begin with the definition of a (pure) Hodge structure:

\begin{definition} A (\emph{pure}) \emph{Hodge structure of weight $n \in \Z$}, denoted $(H_{\Z},H^{p,q})$, consists of a finitely generated free abelian group $H_{\Z}$ (a \emph{lattice}) along with a decomposition $H_{\C} = \bigoplus_{p+q = n} H^{p,q}$ of the complexification  $H_\CC := H_{\Z} \otimes_{\Z} \C$, which satisfies $H^{p,q} = \overline{H^{q,p}}$.
\end{definition}

Note that one can also speak of rational (respectively real) Hodge structures, obtained by replacing the lattice $H_{\Z}$ with a rational (respectively real) vector space. In fact, one can further generalize this notion so as to allow Hodge structures on $R$-modules $H_R$ of finite type, with $R \subset \mathbb{R}$ an arbitrary subring. 

\begin{example}
Defining $H_{\C} = H^{k,k}$ and $H^{p,q} = 0, (p, q) \neq (k, k)$, one obtains the simplest example of a Hodge structure, called a \emph{trivial Hodge structure of weight $2k$}.
\end{example}

\begin{example} Another simple Hodge structure is given by taking $H_{\Z} = 2\pi i \Z$ (considered as a subgroup of $\C$) and setting $H_{\C} = H^{-1,-1}$. This is a pure Hodge structure of weight $-2$ and, in fact, is the unique $1$-dimensional pure Hodge structure of weight $-2$ up to isomorphism. This Hodge structure is called the \emph{Tate Hodge structure} and is often denoted by $\Z (1)$.
\end{example}

It is also common to see a pure Hodge structure of weight $n$ defined by a decreasing filtration $\{F^p\}$ on $H_\CC$
\[ H_\CC = F^0 \supset F^1 \supset \cdots \supset F^n \supset \{0\}\]
such that $H_\CC \cong F^p \oplus \overline{F^{n-p+1}}$. The two definitions are completely equivalent: given a decomposition $H_\CC = \bigoplus_{p+q = n} H^{p,q}$ we may define a filtration by setting $F^p := H^{n,0} \oplus \cdots \oplus H^{p,n-p}$, and given a filtration $\{F^p\}$, we may define a decomposition by setting $H^{p,q} := F^p \cap \overline{F^q}$. The Hodge filtration will prove to be a useful reformulation when we come to study Hodge structures associated to compact K\"{a}hler varieties, as it varies holomorphically in families (see Sect. \ref{sect:locmap}). 
\medskip

To produce different Hodge structures starting from given ones, it is natural\footnote{A third characterization of Hodge structures is given in terms of certain representations of $\mathrm{Res}_{\CC/\RR} \CC^*$ (see, for example, \cite{vG00}). More precisely, a rational Hodge structure of weight $n$ on a $\QQ$-vector space $H$ can be identified with an algebraic representation $\rho\colon \C^* \ra GL(H_\RR)$, where $H_\RR:=H\otimes_\QQ\RR$, such that the restriction of $\rho$ to $\RR^*$ is given by $\rho(\lambda) = \lambda^n$. From this point of view, it is clearly completely natural to use constructions from multi-linear algebra to produce new Hodge structures.} to use the following multi-linear algebra constructions:
\begin{enumerate}
\item Let $(H_{\Z},H^{p,q})$, $(H'_{\Z},H'^{p,q})$ be two Hodge structures, both of weight $n$. Then we can define their direct sum by taking the underlying lattice to be the direct sum of the two lattices $H_{\Z} \oplus H'_{\Z}$, and the $(p, q)$-components to be the direct sums of the $(p, q)$-components of each term $H^{p,q} \oplus H'^{p,q}$. The direct sum is thus a Hodge structure of weight $n$.
\item The dual of a Hodge structure $(H_{\Z},H^{p,q})$ of weight $n$ is a Hodge structure of weight $-n$, defined by taking as underlying lattice the dual $H_\Z^\vee := \Hom (H_\Z, \Z)$ with the dual Hodge decomposition $(H^{\vee})^{p,q} = (H^{-p,-q})^\vee.$ 
\item Let $(H_{\Z},H^{p,q})$, $(H'_{\Z},H'^{p,q})$ be two Hodge structures of weight $n$ and $n'$ respectively. Then we can define their tensor product by taking as underlying lattice $H''_{\Z} = H_{\Z} \otimes H'_{\Z}$ and defining the Hodge decomposition on its complexification as:
\[
H''^{p,q} = \bigoplus_{\substack{r+r' = p\\ s+s'=q}} H^{r,s} \otimes H'^{r',s'}. 
\]
The tensor product is a Hodge structure of weight $n + n'$.
\item From the previous two constructions it immediately follows that, for Hodge structures $(H_{\Z},H^{p,q})$ and $(H'_{\Z},H'^{p,q})$ of weights $n$ and $n'$, we also have Hodge structures on $\Hom(H_\Z , H'_\Z ) = H_\Z^\vee \otimes H'_\Z$  of weight $n'-n$ and on $\mathrm{Sym}^k (H_\Z )$ and $\bigwedge^k H_\Z$, both of weight $kn$.
\end{enumerate}

\begin{example}
Starting from a Hodge structure $(H_{\Z}, H^{p,q})$ of weight $n$, we can produce a new Hodge structure of weight $n - 2r$, which is referred to as the \emph{$r$-th Tate twist} of the original Hodge structure, by setting
\[
H(r)_{\Z} = H_{\Z}, \quad H(r)^{p,q} = H^{p-r,q-r}. 
\]
\end{example}

\subsection{The Hodge Decomposition} \label{sect:hodgedecomp}

One of the main applications of Hodge structures is to the study of the cohomology of K\"{a}hler manifolds, via the \emph{Hodge decomposition}. This decomposition will be described in this section.

Begin by letting $X$ denote an $m$-dimensional Riemannian manifold. Denote the sheaf of smooth $n$-forms on $X$ by $\calA^n_X$ and let $d\colon \calA^n_X \to \calA^{n+1}_X$ be the exterior derivative. Define the \emph{Laplacian}
\[\Delta_d = d\delta + \delta d,\]
where $\delta\colon \calA^n_X \to \calA^{n-1}_X$ is the \emph{codifferential}, defined by $\delta = (-1)^{nm+m+1}*d*$, where $*$ denotes the Hodge star operator.

Next define the set of \emph{harmonic forms of degree $n$} to be
\[\calH^n(X) := \{\alpha \in \calA^n_X \mid \Delta_d\alpha = 0\}.\]
Then we have:

\begin{theorem}[Hodge's Theorem]  \textup{\cite[Thm. 5.23]{htcagI}} There is an isomorphism
\[\calH^n(X) \cong H^n(X,\RR).\]
\end{theorem}

To exploit the full power of Hodge's Theorem, we need to consider a complex manifold $X$ endowed with a Hermitian metric. In this setting we may decompose the sheaf of complex $n$-forms $\calA^n_X$ into a direct sum of sheaves of $(p,q)$-forms
\[\calA^n_X = \bigoplus_{p+q=n} \calA^{p,q}_X\]
and write $d = \partial + \overline{\partial}$, where $\partial\colon \calA^{p,q}_X \to \calA^{p+1,q}_X$ and $\overline{\partial}\colon \calA^{p,q}_X \to \calA^{p,q+1}_X$ are the \emph{Dolbeault operators}. As we defined the Laplacian, we may also define operators $\Delta_{\partial}$ and $\Delta_{\overline{\partial}}$, and these two operators both preserve the bidegree given by the decomposition of complex $n$-forms into $(p,q)$-forms. However it is important to note that, on an arbitrary complex manifold with Hermitian metric, the two operators $\Delta_{\partial}$ and $\Delta_{\overline{\partial}}$ are not necessarily related to the Laplacian $\Delta_d$, and $\Delta_d$ does not necessarily preserve the bidegree. 

To rectify this, we further restrict our attention to \emph{K\"{a}hler manifolds}. A Hermitian metric on a complex manifold $X$ is said to be \emph{K\"{a}hler} if its imaginary part $\omega$, which is a $(1,1)$-form on $X$, is closed. A complex manifold $X$ equipped with a K\"{a}hler metric is called a \emph{K\"{a}hler manifold} and the $2$-form $\omega$ on $X$ is called the associated \emph{K\"{a}hler form}. 

The extraordinariness of K\"{a}hler manifolds from the point of view of Hodge theory relies on the fact that, on a K\"{a}hler manifold, the operator $\Delta_d$ preserves the bidegree. Indeed, via the \emph{Hodge identities} \cite[Prop. 6.5]{htcagI}, we obtain:

\begin{theorem}  \textup{\cite[Thm. 6.7]{htcagI}} If $X$ is a K\"{a}hler manifold, then
\[\Delta_d = 2\Delta_{\partial}= 2\Delta_{\overline{\partial}}.\]
\end{theorem}

As an easy corollary, we find:

\begin{corollary}  \textup{\cite[Cor. 6.9]{htcagI}} If $\alpha$ is a harmonic form of degree $n$ on a K\"{a}hler manifold $X$, then the components of $\alpha$ of type $(p,q)$ are also harmonic.
\end{corollary}

Thus, we may decompose
\[\calH^n(X) = \bigoplus_{p+q=n}\calH^{p,q}(X),\]
where $\calH^{p,q}(X)$ denotes the space of harmonic forms of type $(p,q)$, and this decomposition satisfies $\calH^{p,q}(X) = \overline{\calH^{q,p}(X)}$ \cite[Cor. 6.10]{htcagI}.

Finally, if $X$ is also assumed to be compact, then we have $\calH^n(X) \cong H^n(X,\C)$ and the decomposition of $\calH^n(X)$ into $\calH^{p,q}(X)$ induces a decomposition 
\[H^n(X,\C) = \bigoplus_{p+q=n}H^{p,q}(X).\]
Via Dolbeault's isomorphism, it can be shown that $H^{p,q}(X) \cong H^q(X,\Omega^p_X)$, where $\Omega^p_X$ is the sheaf of holomorphic $p$-forms on $X$ (we refer the interested reader to \cite[Lemma 6.18]{htcagI} for details). We thus find:

\begin{theorem}[Hodge Decomposition] \textup{\cite[Sect. 6.1.3]{htcagI}} \label{hodgedecomp} Let $X$ be a compact K\"{a}hler manifold. Then there exists a decomposition
\[H^n(X,\C) = \bigoplus_{p+q=n}H^{p,q}(X),\]
where $H^{p,q}(X) \cong H^q(X,\Omega^p_X)$ and $H^{p,q}(X) = \overline{H^{q,p}(X)}$. 
\end{theorem}

From this theorem, we see that if we let $H_{\Z}(X):=H^n(X,\Z)/\mathrm{torsion}$, then the data $(H_{\Z}(X),H^{p,q}(X))$ defines a pure Hodge structure of weight $n$. The integers $h^{p,q}(X) = \dim_\C H^{p,q}(X)$ are called the \emph{Hodge numbers} of $X$. Note that the Hodge decomposition implies that $h^{p,q}(X) = h^{q,p}(X)$ and the $n$th Betti number $b_n(X) = \sum_{p+q=n}h^{p,q}(X)$.

The Hodge Decomposition Theorem immediately constrains the cohomology of a K\"{a}hler manifold, as exhibited by the following:

\begin{corollary} \textup{\cite[Cor. 6.13]{htcagI}}
For every compact K\"{a}hler manifold $X$, the odd Betti numbers $b_{2k-1}(X)$ are even.
\end{corollary}

The Hodge numbers of a compact K\"{a}hler manifold $X$ are frequently displayed in the \emph{Hodge diamond}:
\[\begin{array}{ccccccc} & & & h^{0,0}(X) & & & \\
& & h^{1,0}(X) & & h^{0,1}(X) & & \\
&\iddots & & \vdots & & \ddots & \\
h^{m,0}(X) & & h^{m-1,1}(X) & \cdots  & h^{1,m-1}(X) & & h^{0,m}(X) \\
& \ddots & & \vdots & & \iddots & \\
& & h^{m,m-1}(X) & & h^{m-1,m}(X) & & \\
& & & h^{m,m}(X) & & &
\end{array}\]
where $m = \dim_{\C}(X)$.

\begin{example} Let $C$ be a curve of genus $g$. Then it is easy to see that the Hodge diamond of $C$ is given by
\[\begin{array}{ccc} & h^0(C,\calO_C) &  \\
h^0(C,\omega_C) & & h^1(C,\calO_C)  \\
& h^1(C,\omega_C) & 
\end{array}
\quad =\quad
\begin{array}{ccc} & 1 &  \\
g & & g  \\
& 1 & 
\end{array}\]
\end{example}

\begin{example} \label{K3Hodgeex} Now let $S$ be a K3 surface, i.e.\ a smooth compact complex surface with trivial canonical bundle and $h^1(S,\calO_S) = 0$. The Hodge diamond of $S$ is calculated in \cite[Prop. VIII.3.4]{bpv}, giving
\[\begin{array}{ccccc} & & h^0(S,\calO_S) & &  \\
& h^0(S,\Omega^1_S) & & h^1(S,\calO_S) &  \\
h^0(S,\omega_S) & & h^1(S,\Omega^1_S) && h^2(S,\calO_S) \\
& h^1(S,\omega_S) & & h^2(S,\Omega^1_S) & \\
& & h^2(S,\omega_S) & & 
\end{array}
=
\begin{array}{ccccc} & & 1 & & \\
& 0 & & 0 &  \\
1 & & 20  & & 1 \\
& 0 & & 0 & \\
& & 1 & &  
\end{array}\]
\end{example}

\begin{example} \label{CY3Hodgeex} Finally, let $X$ be a Calabi-Yau threefold, i.e.\ a smooth compact complex K\"ahler threefold with trivial canonical bundle and $h^i (X, \calO_X) = 0$ for $0 < i < 3$. From the definition of $X$ and Serre duality, we obtain the following Hodge diamond:
\[\begin{array}{ccccccc}
& & & 1 & & & \\
& & 0 &  & 0 & & \\
& 0\  & & h^{1,1}(X) & &\  0 & \\
1\quad &  & h^{2,1}(X) & & h^{2,1}(X) & &\quad 1\\
& 0\  & & h^{1,1}(X) & &\ 0 & \\
& & 0 & & 0 & & \\
& & & 1 & & &
 \end{array}
\]
There are two Hodge numbers in the centre of this diamond that are not determined by the general definition of a Calabi-Yau threefold. They can be interpreted as follows.

From the vanishing $h^{1}(X,\mathcal{O}_{X})=h^{2}(X,\mathcal{O}_X)=0$ and the exponential sheaf sequence, we get an abelian group isomorphism between $\mathrm{Pic}(X)=H^{1}(X,\mathcal{O}_{X}^{*})$ and $H^{2}(X,\mathbb{Z})$. Thus, by the Hodge decomposition, the Hodge number $h^{1,1}(X)$ is equal to the Picard number $\rho(X) = \dim (\mathrm{Pic}(X))$. 

By Serre duality, we see that $h^{2,1}(X)=h^{1}(X,\mathcal{T}_{X})$, where $\mathcal{T}_X =(\Omega_X)^{*}$ is the tangent bundle of $X$. Thus, by Kodaira-Spencer theory, the Hodge number $h^{2,1}(X)$ is equal to the dimension of the space of first order infinitesimal complex deformations of $X$. 

A well-known consequence of the mirror symmetry conjecture is that, for a mirror pair of Calabi-Yau threefolds $X$ and $\check{X}$, these two Hodge numbers are interchanged:
\[h^{2,1}(X) = h^{1,1}(\check{X}),  \qquad h^{1,1}(X) = h^{2,1}(\check{X}).\]
\end{example}

Note that the symmetries of the Hodge diamonds above are the result of general phenomena: the left-right symmetry is a consequence of the equality $h^{p,q}(X) = h^{q,p}(X)$, whereas the top-bottom symmetry is a consequence of Serre duality, which implies that $h^{p,q}(X)= h^{m-p,m-q}(X)$ for an $m$-dimensional compact K\"{a}hler manifold $X$ (see \cite[Sect. 5.3.2]{htcagI}).

\subsection{Morphisms of Hodge Structures}

We define a morphism of Hodge structures as follows:

\begin{definition} Let $(V_\Z,V^{p,q})$ and $(W_\Z,W^{p,q})$ denote two Hodge structures of weights $n$ and $n + 2r$ respectively (for some $r \in \Z$). Then a \emph{morphism of Hodge structures} of bidegree $(r,r)$ is a group homomorphism $\phi\colon V_\Z \to W_\Z$ such that $\phi(V^{p,q}) \subset W^{p+r,q+r}$ (or, equivalently, in terms of the Hodge filtrations, $\phi(F^p V_\C) \subset F^{p+r}W_\C$).
\end{definition}

\begin{example} 
Let $X$ and $Y$ be two compact K\"{a}hler manifolds and let $f\colon X \to Y$ be a holomorphic map. Then $f^*\colon H^n(Y,\Z) \to H^n(X,\Z)$ is a morphism of Hodge structures of bidegree $(0,0)$. 

The \emph{Gysin morphism} $f_*\colon H^n(X,\Z) \to H^{n-2r}(Y,\Z)$ is also a morphism of Hodge structures of bidegree $(r,r)$, where $r = \dim_\C(Y) - \dim_\C(X)$ \cite[Sect. 7.3.2]{htcagI}.
\end{example}

\begin{remark}
For every Hodge structure $(H_\Z, H^{p,q})$ of weight $2k-1$ with associated Hodge filtration $F^{\bullet}H_{\C}$ we can define a complex torus  by
\[
J^k(H) = \frac{H_\C}{F^k H_\C \oplus H_\Z}.
\]
Given a morphism of Hodge structures, we get an induced  morphism of complex tori. Appying this construction to $H^{2k-1}(X, \C)$, where $X$ is a compact K\"ahler manifold, we obtain, for each $k>0$, the \emph{Intermediate Jacobian}:
\[
J^k(X) = \frac{H^{2k-1}(X,\C)}{F^k H^{2k-1}(X, \C) \oplus H^k(X, \Z)},
\]
which is indeed a complex torus \cite[Sect. 12.1]{htcagI}.
\end{remark}

The following result tells us how the Hodge filtration behaves under morphisms of Hodge structures:

\begin{lemma} \textup{\cite[Lemma 7.23]{htcagI}}
A morphism of Hodge structures $\phi\colon V_\Z \to W_\Z$ is strict for the Hodge filtration, i.e. $\im \phi \cap F^{p+r}W_{\C}= \phi(F^{p}V_{\C}).$
\end{lemma}

With the help of this lemma we can define a Hodge structure on the image of a morphism of Hodge structures: One takes the quotient filtration induced by $F^\bullet V_{\C}$, which coincides with the one inherited by $\im \phi$ from $F^{\bullet + r} W_{\C}$ \cite[Cor. 7.24]{htcagI}. It can further be shown that the kernels and cokernels of morphisms of Hodge structures are indeed Hodge (sub-)structures defined by the induced filtration \cite[Sect. 7.3.1]{htcagI}.

Restricting to the case of rational Hodge structures, we have the following:
\begin{definition}
A rational \emph{Hodge substructure} of $(H_{\mathbb{Q}}, H^{p,q})$ is a $\mathbb{Q}$-vector subspace $W_{\mathbb{Q}}$ of $H_{\mathbb{Q}}$ such that the decomposition of its complexification $W_{\C} = W_{\mathbb{Q}} \otimes_{\mathbb{Q}} \C$ satisfies:
\[W_{\C} = \bigoplus_{p+q=n} (W_{\C} \cap H^{p,q})\]
\end{definition}
The Hodge substructure on $W$ is said to be of weight $l$ if $F^{l+1} W = 0$. In particular, a Hodge substructure can have lower weight than the Hodge structure it is contained in.

\begin{remark}
With the help of the previous results on morphisms and the operations on Hodge structures defined above, it can be seen that the category of rational Hodge structures of a given weight, with morphisms given by morphisms of Hodge structures of type $(0,0)$, is an abelian category \cite[Sect. 7.3.1]{htcagI}.
\end{remark}

\subsection{The Lefschetz Decomposition}

In this section, we digress briefly to discuss a second important decomposition of the cohomology of a compact K\"{a}hler manifold: the \emph{Lefschetz decomposition}. Suppose that $X$ is a compact K\"{a}hler manifold and let $\eta \in H^k(X,\Z)$. Then $\eta$ induces a map $\eta\colon H^n(X,\Z) \to H^{k+n}(X,\Z)$, via the cup-product.

Let $\omega$ denote the K\"{a}hler form on $X$. Then $[\omega] \in H^2_{\mathrm{dR}}(X)$, the second de Rham cohomology of $X$. The cup-product with $\omega$ thus induces a map
\[L \colon H^n(X,\RR) \longrightarrow H^{n+2}(X,\RR).\]
We have:

\begin{theorem}[Hard Lefschetz] \textup{\cite[Sect. 6.2.3]{htcagI}} Let $X$ be a compact K\"{a}hler manifold of dimension $m$. Then
\[L^{m-n}\colon H^n(X,\RR) \longrightarrow H^{2m-n}(X,\RR)\]
is an isomorphism. Furthermore, if $n \leq j \leq m$, then
\[L^{m-j}\colon H^n(X,\RR) \longrightarrow H^{2m+n-2j}(X,\RR)\]
is injective.
\end{theorem}

The Hard Lefschetz Theorem provides us with further constraints on the topology of K\"{a}hler manifolds:

\begin{corollary} The odd Betti numbers $b_{2k-1}(X)$ increase for $2k - 1 < n$ and, similarly, the even Betti numbers $b_{2k}(X)$ increase for $2k < n$.
\end{corollary}

Next, define:

\begin{definition} Let $X$ be a compact K\"{a}hler manifold of dimension $m$. Define the \emph{$n$th primitive cohomology of $X$} by
\[P^n(X,\RR) := \ker(L^{m-n+1}\colon H^n(X,\RR) \longrightarrow H^{2m-n+2}(X,\RR)).\]
\end{definition}

Then we have:

\begin{theorem}[Lefschetz Decomposition] \textup{\cite[Cor. 6.26]{htcagI}} Let $X$ be a compact K\"{a}hler manifold of dimension $m$. Then there is a decomposition
\[H^n(X,\RR) = \bigoplus_{2r \leq n} L^rP^{n-2r}(X,\RR).\]
\end{theorem}

Furthermore, this decomposition is compatible with the Hodge decomposition, so that if we write $P^{p,q}(X) = P^n(X,\C) \cap H^{p,q}(X)$, where $n = p+q$ and $P^n(X,\C) = P^n(X,\RR) \otimes \C$, then the Hodge decomposition induces a decomposition $P^n(X,\C) = \bigoplus_{p+q=n} P^{p,q}(X)$.
We refer the interested reader to \cite[Rem. 6.27]{htcagI} for details.

\subsection{Spectral Sequences}\label{specseqsec}

Next we discuss an important spectral sequence and show how it can be used to deduce a weaker version of the Hodge decomposition theorem. These ideas will be revisited in the context of mixed Hodge structures in Section \ref{sect:mhsbc}.

We begin with some background on hypercohomology and spectral sequences; a more detailed discussion may be found in \cite[Ch. 8]{htcagI}. Let 
\[\cdots \stackrel{d}{\longrightarrow} A^0\stackrel{d}{\longrightarrow} A^1 \stackrel{d}{\longrightarrow} A^2 \stackrel{d}{\longrightarrow} \cdots\]
be a complex of sheaves on a space $X$. Recall that the \emph{hypercohomology} of the complex $(A^\bullet,d)$ is defined by choosing an acyclic resolution of $A^\bullet$ by a double complex $(I^{\bullet,\bullet},d,d')$, i.e. a diagram with exact rows and columns
\begin{center}
\centerline{
\xymatrix@C=30pt
{
 & \vdots&\vdots&\vdots\\
\cdots\ar[r] & I^{0,1} \ar_d[r]\ar_{d'}[u]& I^{1,1} \ar_d[r]\ar_{d'}[u]& I^{2,1} \ar_d[r]\ar_{d'}[u]&\cdots\\
\cdots\ar[r] & I^{0,0} \ar_d[r]\ar_{d'}[u]& I^{1,0} \ar_d[r]\ar_{d'}[u]& I^{2,0} \ar_d[r]\ar_{d'}[u]&\cdots\\
\cdots\ar[r] & A^0 \ar_d[r]\ar_{d'}[u]& A^1 \ar_d[r]\ar_{d'}[u]& A^2 \ar_d[r]\ar_{d'}[u]&\cdots\\
 & 0 \ar[u] & 0 \ar[u] & 0 \ar[u]
}}
\end{center}
and $H^n(X,I^{i,j})=0$ for $n\ge 0$. Such a resolution always exists. Let $I^n :=\bigoplus_{i+j=n} I^{i,j}$ denote the total complex with differential $\delta=d+(-1)^id'$, then the $n$th hypercohomology of $(A^\bullet,d)$ is defined to be the $n$th cohomology of the complex of global sections of $I^{\bullet}$,
\[\HH^n(X,A^\bullet):=H^n_\delta \Gamma(X,I^\bullet),\]
and this definition is independent of the choice of $I^{\bullet,\bullet}$.

Next assume that we have a decreasing filtration $F$ on $(A^\bullet,d)$ turning it into a filtered complex, i.e. a decreasing filtration 
\[\cdots \subset F^2A^k \subset F^1A^k \subset F^0A^k=A^k\] 
on each $A^k$, such that $d$ preserves the filtration  $d\colon F^pA^k\ra F^pA^{k+1}$. Replacing $I^{\bullet,\bullet}$ if necessary, we may assume that the filtration $F$ lifts to the resolution, in other words, that we have a filtration 
\[ I^{i,j}=F^0I^{i,j} \supset F^1I^{i,j} \supset F^2I^{i,j} \supset \cdots\] 
on the double complex $I^{\bullet,\bullet}$ that is compatible with differentials, such that $F^iI^{\bullet,\bullet}$ is an acyclic resolution of $F^iA^\bullet$. As before, we denote the total complex of $F^iI^{\bullet,\bullet}$ by $F^iI^{\bullet}$ and denote its differential by $\delta$. 

We may use this filtration $F$ to define a filtration on the hypercohomology of $A^\bullet$ as follows. The embedding $F^iI^{\bullet}\subseteq I^{\bullet}$ induces a map of hypercohomologies
\[\HH^n(X,F^pA^\bullet)\lra\HH^n(X,A^\bullet).\]
We may then simply define $F^p\HH^n(X,A^\bullet)$ to be the image of this map.

Now suppose that the filtration $F^pA^{\bullet}$ is bounded, so that for each $k$, there exists a $p$ with $F^pA^k = 0$. Then we have:

\begin{theorem}\label{thm:specseq} \textup{\cite[Thm. 8.21]{htcagI}}   There exist complexes 
\[(E^{p,q}_r,d_r), \quad \mbox{with differentials } \, d_r\colon E^{p,q}_r \rightarrow E^{p+r,q-r+1}_r\]
which satisfy the following conditions:
\begin{enumerate}
\item $E_0^{p,q} = \Gamma(X,\,\Gr_F^p I^{p+q}) := \Gamma(X,\,F^p I^{p+q}/F^{p+1} I^{p+q})$ and $d_0$ is induced by $\delta$.
\item $E^{p,q}_{r+1}$ can be identified with the cohomology of $(E_r^{p,q},d_r)$, i.e. with
\[\frac{\mathrm{ker}(d_r\colon E^{p,q}_r \to E^{p+r,q-r+1}_r)}{\im(d_r\colon E^{p-r,q+r-1} \to E^{p,q}_r)}.\]
\item For $p + q$ fixed and $r$ sufficiently large,
\[E^{p,q}_r = \Gr_F^p \HH^{p+q}(X,A^\bullet)\]
\end{enumerate}
\end{theorem}

We note here that the exactness of $\Gamma$ on acyclic objects implies that $\Gamma$ commutes with $\Gr_F$, so in part 1 of the above theorem we have
\[\Gamma\left(X,\Gr_F^p I^{p+q}\right) = \frac{\Gamma(X,\,F^p I^{p+q})}{\Gamma(X,\,F^{p+1} I^{p+q})}.\]

An explicit definition of $E^{p,q}_r$ may be given by defining $Z_r^{p,q}$ and $B_r^{p,q}$ to be
\begin{eqnarray*}
Z_r^{p,q} &:=& \ker(\delta \colon \Gamma(X,\,F^p I^{p+q})\ra \Gamma(X,\,I^{p+q+1}/F^{p+r}I^{p+q+1})),\\
\Gamma(X,\,I^{p+q})/B_r^{p,q} &:=& \coker(\delta \colon \Gamma(X,\,F^{p-r+1} I^{p+q-1})\ra \Gamma(X,\,I^{p+q}/F^{p+1}I^{p+q})),
\end{eqnarray*}
then setting
\[E^{p+q}_r=Z_r^{p,q}/(B_r^{p,q}\cap Z_r^{p,q});\]
we refer the interested reader to the proof of \cite[Thm. 8.21]{htcagI}.

The series of complexes $(E^{p,q}_r,d_r)$ is called the \emph{spectral sequence} associated to the filtered complex $(A^{\bullet},F)$. Part 3 of Thm. \ref{thm:specseq} says that this spectral sequence \emph{converges} to $\HH^{p+q}(X,A^\bullet)$. This group is often denoted by $E^{p+q}_{\infty}$, and its graded piece $\Gr_F^p \HH^{p+q}(X,A^\bullet)$ denoted by $E^{p,q}_{\infty}$. We write
\[ E_0^{p,q}=\Gamma(X,\Gr_F^p I^{p+q}) \Rightarrow \HH^{p+q}(X,A^\bullet),\]
which should be read as ``the spectral sequence with $E_0^{p,q}=\Gamma(X,\Gr_F^p I^{p+q})$ converges to $ \HH^{p+q}(X,A^\bullet)$''. Note here that it is also not uncommon to see a spectral sequence defined using $E_r^{p,q}$, for some $r>0$, on the left hand side instead of $E_0^{p,q}$.

We say that the spectral sequence associated to a filtered complex $(A^\bullet, F)$ \emph{degenerates at $E_r$} if $d_k = 0$ for all $k \geq r$. For such $r$, from Thm. \ref{thm:specseq} we obtain 
\[ E_r^{p,q} = E_\infty^{p,q} = \Gr_F^p \HH^{p+q}(X,A^\bullet).\]

Now we specialize this discussion to the case that interests us. Begin by letting $X$ be a complex manifold (that is not necessarily K\"{a}hler). With notation as in Sect. \ref{sect:hodgedecomp}, we consider the holomorphic de Rham complex $(\Omega^\bullet_X, \partial)$. Equip this complex with the naive filtration given by truncation 
\[F^p \Omega^\bullet_X = \Omega^{\geq p}_X = 0 \longrightarrow \cdots \longrightarrow 0 \longrightarrow \Omega^p_X \stackrel{\partial}{\longrightarrow} \Omega^{p+1}_X \stackrel{\partial}{\longrightarrow} \cdots. \]
We will use the theory above to associate a spectral sequence to this filtered complex, which will allow us to give a weaker form of the Hodge decomposition theorem.

To define the spectral sequence that we need, note first that the double complex $(\calA^{p,q}_X,\partial,(-1)^p\overline{\partial})$ (as defined in Sect. \ref{sect:hodgedecomp}) provides an acyclic resolution of $\Omega^{\bullet}_X$
\begin{center}
\centerline{
\xymatrix@C=30pt
{
 & \vdots&\vdots&\vdots\\
0 \ar[r] & \calA_X^{0,1} \ar_{\partial}[r]\ar_{\overline{\partial}}[u]& \calA_X^{1,1} \ar_{\partial}[r]\ar_{-\overline{\partial}}[u]& \calA_X^{2,1} \ar_{\partial}[r]\ar_{\overline{\partial}}[u]&\cdots\\
0 \ar[r] & \calA_X^{0,0} \ar_{\partial}[r]\ar_{\overline{\partial}}[u]& \calA_X^{1,0} \ar_{\partial}[r]\ar_{-\overline{\partial}}[u]& \calA_X^{2,0} \ar_{\partial}[r]\ar_{\overline{\partial}}[u]&\cdots\\
0 \ar[r] &  \Omega^0_X \ar_{\partial}[r]\ar@{^(->}[u]& \Omega_X^1 \ar_{\partial}[r]\ar@{^(->}[u]& \Omega_X^2 \ar_{\partial}[r]\ar@{^(->}[u]&\cdots\\
 & 0 \ar[u] & 0 \ar[u] & 0 \ar[u]
}}
\end{center}
and the associated total complex is precisely the de Rham complex $\calA_X^{\bullet}$ with the exterior derivative $d = \partial + \overline{\partial}$.

The filtration $F^p\Omega^\bullet_X$ lifts to a filtration on the double complex $\calA_X^{\bullet,\bullet}$, again given by truncation, which induces the filtration $F^p\calA^{\bullet}_X$ defined on the de Rham complex $(\calA^{\bullet}_X,d)$ by
\[F^p\calA^n_X := \bigoplus_{\substack{i \geq p \\ i+j = n}} \calA^{i,j}_X.\]

This filtration induces a filtration on the hypercohomology of $(\Omega_X^{\bullet},\partial)$, as described above. Moreover, by \cite[Cor. 8.14]{htcagI}, the hypercohomology of this complex agrees with the cohomology of $X$, i.e. $\mathbb{H}^n(X,\Omega_X^{\bullet}) = H^n(X,\C)$. So we can define an analogue of the Hodge filtration on the cohomology of $X$ using the filtration on the hypercohomology
\[F^p H^n(X, \C) := F^p\mathbb{H}^n (X, \Omega^\bullet_X).\]

Furthermore, by Thm. \ref{thm:specseq}, the filtered complex $(\calA_X^{\bullet},F)$ has an associated spectral sequence
\[E_0^{p,q} = \Gamma(X,\calA_X^{p,q}) \Rightarrow \Gr_F^p \HH^{p+q}(X,\Omega_X^{\bullet}) = \frac{F^pH^{p+q}(X,\C)}{F^{p+1}H^{p+q}(X,\C)}.\] 
This spectral sequence is called the 
\emph{Fr\"{o}licher spectral sequence}.

Now suppose that the Fr\"{o}licher spectral sequence degenerates at $E_1$. It follows from \cite[Prop. 8.25]{htcagI} that $E_1^{p,q} = H^q(X,\Omega_X^p)$, so we thus obtain the equality
\[H^q(X,\Omega_X^p) = F^pH^{p+q}(X,\C)\, / \, F^{p+1}H^{p+q}(X,\C).\]
This implies the existence of a decomposition
\[H^n(X,\C) \cong \bigoplus_{p+q=n} H^q (X, \Omega_X^p ),\]
but this isomorphism is not necessarily canonical. This is a weaker form of the Hodge decomposition. 

\begin{remark}Unfortunately, this weaker result does not imply the equality of the Hodge numbers $h^{p,q} = h^{q,p}$, nor the Hodge decomposition in the usual form
\[H^n(X,\C) = \bigoplus_{p+q=n}H^{p,q}(X), \ \mathrm{where} \ H^{p,q}(X) := F^pH^{p+q}(X,\C) \cap \overline{F^{q+1}H^{q}(X,\C)}.\]
\end{remark}

In the K\"{a}hler case, the Hodge decomposition (Thm. \ref{hodgedecomp}) and Dolbeault's isomorphism $H^{p,q}(X) \cong H^q(X,\Omega^p_X)$ imply the existence of this weaker decomposition, so the following result can be expected.

\begin{theorem} \textup{\cite[Thm. 8.28]{htcagI}} The Fr\"{o}licher spectral sequence of a compact K\"ahler manifold degenerates at $E_1$.
\end{theorem}

\subsection{Polarized Hodge Structures}

We conclude Sect. \ref{sect:Hodgestructures} by defining an important concept: that of a \emph{polarized Hodge structure}. Imposing the additional condition of a polarization will later allow us to classify polarized Hodge structures by points in a \emph{period domain}, which will in turn prove to be a very powerful tool in the study of Hodge structures that vary in families (see Sect. \ref{sect:vhs}).

We begin by letting $X$ be a compact K\"{a}hler manifold with K\"{a}hler form $\omega$. Fix once and for all an integer $n \geq 0$. Let $H_{\Z}:=H^n(X,\Z)/\mathrm{torsion}$ and $H_\CC := H_{\Z}\otimes_{\Z}\C \cong H^n(X,\C)$.

The Hodge decomposition (Thm. \ref{hodgedecomp}) shows that we may decompose $H_\CC$ as a direct sum
\[H_\CC = \bigoplus_{p+q=n} H^{p,q},\]
so that the data $(H_{\Z},H^{p,q})$ defines a pure Hodge structure of weight $n$.

Now, we can use $\omega$ to define a nondegenerate bilinear form $Q\colon H_{\Z} \times H_{\Z} \to \Z$ by
\[ Q(\xi,\eta) := \int_X \xi \wedge \eta \wedge \omega^{\dim(X)-n}.\]

$Q$ extends to $H_\CC$ by linearity and has the following properties \cite[Sect. 7.1.2]{htcagI}:
\begin{enumerate}
\item $Q$ is symmetric if $n$ is even and skew-symmetric if $n$ is odd.
\item $Q(\xi,\eta) = 0$ for $\xi \in H^{p,q}$ and $\eta \in H^{p',q'}$ with $p \neq q'$.
\item $(-1)^{\frac{n(n-1)}{2}} i^{p-q} Q(\xi,\overline{\xi}) > 0$ for $\xi \in H^{p,q}$ non-zero.
\end{enumerate}
Conditions (2) and (3) are called the \emph{Hodge-Riemann bilinear relations}.

This motivates the following definition:

\begin{definition} A \emph{polarized Hodge structure of weight $n$} consists of a pure Hodge structure $(H_{\Z},H^{p,q})$ of weight $n$ together with a nondegenerate integral bilinear form $Q$ on $H_{\Z}$ which extends to $H_\CC$ by linearity and satisfies (1)--(3) above.
\end{definition}

\begin{remark} It is also common to see the Hodge-Riemann bilinear relations written in terms of the filtration $\{F^p\}$. In this case they become:
\begin{enumerate}[(1')]\addtocounter{enumi}{1}\setlength{\itemindent}{\parindent}
\item $Q(F^p,F^{n-p+1}) = 0$.
\item $(-1)^{\frac{n(n-1)}{2}} Q(C\xi,\overline{\xi}) > 0$ for any nonzero $\xi \in H_\CC$, where $C \colon H_\CC \to H_\CC$ is the \emph{Weil operator} defined by $C|_{H^{p,q}} = i^{p-q}$.
\end{enumerate}
\end{remark}

To illustrate the theory presented here, we will discuss three examples. These examples are studied in depth in the book by Barth, Hulek, Peters and van de Ven \cite{bpv}; we will return to them repeatedly in Sect. \ref{sect:vhs}.

\begin{example} \label{example-ell-curve} For our first example, let $E$ denote an elliptic curve. We will study the polarized Hodge structure of weight $1$ on the first cohomology $H^1(E,\Z)$.

Define $H_{\Z} := H^1(E,\Z) \cong \Z^2$ and $H_\CC := H^1(E,\C) \cong \C^2$. The Hodge decomposition (Thm. \ref{hodgedecomp}) implies that we may write
\[ H_\CC = H^{1,0} \oplus H^{0,1},\]
where $H^{p,q} = H^q(E,\Omega^p_E)$ and $\dim (H^{1,0}) = \dim (H^{0,1}) = 1$. The data $(H_{\Z},H^{p,q})$ defines a pure Hodge structure of weight $1$.

The polarization on $H_\CC$ is defined by
\[ Q(\xi,\eta) := \int_E \xi \wedge \eta\]
and there exists a canonical basis $\alpha,\beta \in H^1(E,\Z)$ (given by taking the Poincar\'{e} dual of the canonical basis for $H_1(E,\Z)$) so that the matrix of $Q$ with respect to this basis is
\[\left[ \begin{array}{cc} 0 & 1 \\ -1 & 0 \end{array} \right],\]
i.e. $Q(\alpha,\alpha) = Q(\beta,\beta) = 0$ and $Q(\alpha,\beta) = 1$.
\end{example}

\begin{example} We can generalize this result to a arbitrary curve $C$ of genus $g \geq 1$. As before, define $H_{\Z} := H^1(C,\Z) \cong \Z^{2g}$ and $H_\CC:= H^1(C,\C) \cong \C^{2g}$. Again, the Hodge decomposition (Thm. \ref{hodgedecomp}) implies that we may write
\[ H_\CC = H^{1,0} \oplus H^{0,1},\]
where $H^{p,q} = H^q(C,\Omega^p_C)$ and $\dim (H^{1,0}) = \dim (H^{0,1}) = g$. The data $(H_{\Z},H^{p,q})$ defines a pure Hodge structure of weight $1$.

The polarization on $H_\CC$ is again defined by
\[ Q(\xi,\eta) := \int_C \xi \wedge \eta\]
and there exists a canonical basis $\alpha_1,\ldots,\alpha_g,\beta_1,\ldots,\beta_g \in H^1(C,\Z)$ so that the matrix of $Q$ with respect to this basis is
\[\left[ \begin{array}{cc} 0 & I_g \\ -I_g & 0 \end{array} \right] \]
where $I_g$ denotes the $g \times g$ identity matrix, i.e. $Q(\alpha_i,\alpha_j) = Q(\beta_i,\beta_j) = 0$ for all $i,j$ and $Q(\alpha_i,\beta_j) = \delta_{ij}$ (where $\delta_{ij}$ is defined to equal $1$ if $i = j$ and $0$ otherwise).
\end{example}

\begin{example} For our final example, consider a K3 surface $S$. This time, we will study the polarized Hodge structure of weight $2$ on the second cohomology $H^2(S,\Z)$.

Define $H_{\Z} := H^2(S,\Z)$ and $H_\CC := H^2(S,\C)$. The Hodge decomposition (Thm. \ref{hodgedecomp}) gives
\[ H_\CC = H^{2,0} \oplus H^{1,1} \oplus H^{0,2},\]
where $H^{p,q} = H^q(E,\Omega^p_E)$. In this case, Ex. \ref{K3Hodgeex} gives $\dim (H^{2,0}) = \dim (H^{0,2}) = 1$ and $\dim (H^{1,1}) = 20$, and the data $(H_{\Z},H^{p,q})$ defines a pure Hodge structure of weight $2$.

The polarization on $H_\CC$ is defined by
\[ Q(\xi,\eta) := \int_S \xi \wedge \eta.\]
This defines a lattice structure on $H_{\Z} \cong \Z^{22}$. By \cite[Prop. VIII.3.3]{bpv}, the lattice thus obtained is an even, unimodular lattice of rank $22$ and signature $(3,19)$, isometric to 
\[ \Lambda_{\mathrm{K3}} := U \oplus U \oplus U \oplus (-E_8) \oplus (-E_8),\]
where $U$ denotes the hyperbolic plane lattice, an even, unimodular, indefinite lattice of rank $2$ with bilinear form given by the matrix $\left[ \begin{array}{cc} 0 & 1 \\ 1  & 0 \end{array} \right]$, and $E_8$ is the root lattice corresponding to the Dynkin diagram $E_8$, an even, unimodular, positive definite lattice of rank $8$.
\end{example}

\section{Variations of Hodge Structure}\label{sect:vhs}

Our next aim is to make rigorous the idea of polarized Hodge structures that vary in families. This will lead us naturally to the concept of a \emph{polarized variation of Hodge structure} and the \emph{period map} associated to it. Motivated by this definition, we will then be able to define a more general \emph{variation of Hodge structure}. 

Our starting point and main motivation in studying this theory is to understand what happens to the Hodge structure on the cohomology of a K\"{a}hler manifold as that manifold is deformed in a family. This has many uses in the study of Calabi-Yau varieties: in particular, it is crucial to the construction of the moduli space of K3 surfaces (see \cite{gmk3s}) and will be one of the foundations of the Gross-Siebert approach to mirror symmetry, as discussed later in this survey.

Much of this theory was originally developed by Griffiths in the late 1960's in his seminal papers \cite{Gr68a} \cite{Gr68b}, \cite{Gr69} \cite{Gr70b}. As our main reference for many results in this section, we will refer to Part III of the book by Voisin \cite[Chaps. 9-10]{htcagI}. However, for results on period mappings we will sometimes instead refer to the book by Carlson, M\"{u}ller-Stach and Peters \cite{pmpd}, which gives a more comprehensive treatment.  Further details of the three examples presented here may be found in \cite{bpv}.

\subsection{The Local Period Mapping} \label{sect:locmap}

Let $f\colon \calX \to \Delta$ be a proper smooth surjective morphism onto a complex polydisc $\Delta$, whose fibres $X_b$ are compact K\"{a}hler manifolds for all $b \in \Delta$. Assume that there exists $\omega \in H^2(\calX,\Z)$ such that, for each $b \in \Delta$, the restriction $\omega_b := \omega|_{X_b}$ is a K\"{a}hler class. This induces a polarized Hodge structure on the cohomology of the fibres $H^n(X_b,\Z)$ as above, which varies with $b \in \Delta$. This is an example of a \emph{variation of polarized Hodge structure}. By studying it in the forthcoming sections, we will be lead naturally to a general definition for such objects.

Note that smoothness implies that the fibres $X_b$ are all diffeomorphic, and Ehresmann's theorem \cite[Thm. 9.3]{htcagI} shows that $f\colon \calX \to \Delta$ has a local topological trivialization, i.e. $\calX$ is diffeomorphic to $X_b \times \Delta$ over $\Delta$, for any $b \in \Delta$. Thus, there is a \emph{unique} isomorphism $H^n(X_b,\Z) \cong H^n(X_{b'},\Z)$ for any $b,b'\in\Delta$. Therefore, without ambiguity we may define $H_{\Z} := H^n(X_b,\Z)/\mathrm{torsion}$ and $H_\CC := H_{\Z} \otimes_{\Z} \C \cong H^n(X_b,\C)$, where these definitions do not depend upon the choice of $b \in \Delta$. The class $\omega$ induces a bilinear form $Q$ on $H_{\Z}$ as above, which extends to $H_\CC$ by linearity.

However, the isomorphisms $H^n(X_b,\C) \cong H^n(X_{b'},\C)$ \emph{do not} preserve the Hodge decompositions of these spaces; instead, the Hodge decomposition of $H^n(X_b,\C)$ varies continuously with $b \in \Delta$. In particular, although the subspaces $H^{p,q}$ arising from the Hodge decomposition vary with $b \in \Delta$, their dimensions $h^{p,q} := \dim(H^{p,q})$ are fixed. As the $X_b$ are all diffeomorphic, this variation of the Hodge decomposition may be thought of as reflecting a variation of complex structure on a fixed manifold.

We can thus define a classifying space for these Hodge decompositions:

\begin{definition} \label{Ddefn} Let $\calD$ denote the set of all collections of subspaces $\{H^{p,q}\}$ of $H_\CC$ such that $H_\CC = \bigoplus_{p+q=n} H^{p,q}$ and $\dim (H^{p,q}) = h^{p,q}$, on which $Q$ satisfies the Hodge-Riemann bilinear relations (2) and (3).
\end{definition}

\begin{remark} In terms of filtrations, $\calD$ may be defined as the set of all filtrations 
\[ H_\CC =F^0 \supset F^1 \supset \cdots \supset F^n \supset \{0\},\]
with $\dim(F^p) = h^{n,0} + \cdots + h^{p,n-p}$, on which $Q$ satisfies (2') and (3').
\end{remark}

$\calD$ is called the \emph{local period domain}. It can be realized as a homogeneous domain $G/K$, where $G$ is the (real) Lie group of linear automorphisms of $H_{\RR} := H_{\ZZ} \otimes_{\ZZ} \RR$ which preserve $Q$, and $K$ is the subgroup of elements fixing a reference structure in $\calD$ (see \cite[Sect. 4.3]{pmpd}). $\calD$ is thus a (real) manifold.

We may associate to $\calD$ a second manifold $\check{\calD}$, called its \emph{compact dual}, which is defined to be the set of all collections of subspaces $\{H^{p,q}\}$ as in Defn. \ref{Ddefn} that satisfy the first Hodge-Riemann bilinear relation (2) but not necessarily the second (3). It can be shown that $\check{\calD}$ is not just a smooth complex manifold, but also a projective algebraic variety. Moreover, the local period domain $\calD$ may be embedded into $\check{\calD}$ as an open subset, thereby endowing $\calD$ with the structure of a smooth complex manifold. We refer the interested reader to \cite[Sect. 4.3]{pmpd}  for details.

There is a well-defined morphism $\phi\colon \Delta \to \calD$, where $\phi$ takes $b \in \Delta$ to the point in $\calD$ corresponding to the Hodge decomposition of $H^n(X_b,\C)$. This morphism is called the  \emph{local period mapping}.

If $\{F^p_b\}$ is the Hodge filtration on $H^n(X_b,\C)$, we find that $\{F^p_b\}$ has the following properties \cite[Sect. 4.4]{pmpd}:
\begin{align*}
\frac{\partial F^p_b}{\partial \overline{b}} &\subset F^p_b && \mathrm{(\emph{holomorphicity})},\\
\frac{\partial F^p_b}{\partial {b}} &\subset F^{p-1}_b && \mathrm{(\emph{Griffiths transversality})}.
\end{align*}
In particular, the first of these properties implies that the local period mapping $\phi$ is holomorphic \cite[Thm. 4.4.5]{pmpd}.

To illustrate these ideas, we now compute the local period domains for the three examples studied in the previous section.

\begin{example} Consider first the example of an elliptic curve $E$. The polarized Hodge structure on the first cohomology $H^1(E,\Z)$ has Hodge numbers $h^{1,0} = h^{0,1} = 1$. The Hodge filtration is
\[ H = F^0 \supset F^1 \supset \{0\},\]
where $F^1 = H^{1,0}$. We see that $\calD$ is the set of all filtrations $\C^2 \supset F^1 \supset \{0\}$ with $\dim(F^1) = 1$, on which $Q$ satisfies the conditions (2') and (3').

To specify a point in $\calD$, it suffices to give $\lambda \in H_\CC$ that spans $F^1$. The Hodge-Riemann bilinear relations (2') and (3') become
\begin{enumerate}[(1')]\addtocounter{enumi}{1}\setlength{\itemindent}{\parindent}
\item $Q(\lambda,\lambda) = 0$,
\item $i Q(\lambda,\overline{\lambda}) > 0$.
\end{enumerate}
Write $\lambda$ in terms of the canonical basis as $\lambda = z_1\alpha + z_2 \beta$, for $z_1,z_2 \in \C$. The relations become
\begin{enumerate}[(1')]\addtocounter{enumi}{1}\setlength{\itemindent}{\parindent}
\item $z_1z_2 - z_2z_1 = 0$,
\item $i (z_1\overline{z_2} - z_2\overline{z_1} ) > 0$.
\end{enumerate}
(2') is vacuous in this case and (3') implies that $z_1 \neq 0$. We may therefore scale $\lambda$ so that $\lambda = \alpha + z_2 \beta$ (i.e. set $z_1 = 1$). Then (3') says
\begin{enumerate}[(1')]\addtocounter{enumi}{2}\setlength{\itemindent}{\parindent}
\item $\mathfrak{Im} (z_2) > 0$.
\end{enumerate}
Since specifying $z_2$ is equivalent to specifying $\lambda$, which uniquely determines $F^1$, we find that in this case the local period domain for elliptic curves is the complex upper half-plane
\[\calD \cong \mathfrak{H} := \{z \in \C \mid \mathfrak{Im}(z) > 0\}.\]
\end{example}

\begin{example} Now consider the case of a curve $C$ of genus $g \geq 1$. This time the polarized Hodge structure on the first cohomology $H^1(C,\Z)$ has Hodge numbers $h^{1,0} = h^{0,1} = g$. As in the case of the elliptic curve, the Hodge filtration is
\[ H_\CC = F^0 \supset F^1 \supset \{0\},\]
where $F^1 = H^{1,0}$. We see that $\calD$ is the set of all filtrations $\C^2 \supset F^1 \supset \{0\}$ with $\dim(F^1) = g$, on which $Q$ satisfies the Hodge-Riemann bilinear relations (2') and (3').

To specify a point in $\calD$, it is sufficient to give a basis $\{\lambda_1,\ldots,\lambda_g\}$ for the subspace $F^1$. Write $(\lambda_1,\ldots,\lambda_g)$ in terms of the canonical basis $\{\alpha_i,\beta_j\}$ as $(\lambda_1,\ldots,\lambda_g) = (\alpha_1,\ldots,\alpha_g,\beta_1,\ldots,\beta_g)Z$, where $Z = \left[\begin{array}{c} Z_1 \\ Z_2 \end{array} \right]$ is a $2g \times g$ complex matrix. Relations (2') and (3') may then be written in terms of the $g \times g$ matrices $Z_1$ and $Z_2$ as follows:
\begin{enumerate}[(1')]\addtocounter{enumi}{1}\setlength{\itemindent}{\parindent}
\item $Z_1^TZ_2 - Z_2^TZ_1 = 0$, i.e. the matrix $Z_1^TZ_2$ is symmetric, and
\item $i (Z_1^T\overline{Z_2} - Z_2^T\overline{Z_1} )$ is a positive-definite matrix.
\end{enumerate}

In particular, (3') here implies that $Z_1$ is invertible. We may therefore replace $(\lambda_1,\ldots,\lambda_g) \mapsto (\lambda_1,\ldots,\lambda_g)Z_1^{-1}$, so that $Z$ becomes $Z = \left[\begin{array}{c} I_g \\ Z_2 \end{array} \right]$. Then the conditions above become
\begin{enumerate}[(1')]\addtocounter{enumi}{1}\setlength{\itemindent}{\parindent}
\item $Z_2^T - Z_2 = 0$, i.e. the matrix $Z_2$ is symmetric, and
\item $\mathfrak{Im} (Z_2)$ is positive definite.
\end{enumerate}
Thus, we find that in this case the local period domain for genus $g$ curves is the \emph{Siegel upper half-space of degree $g$}
\[\calD \cong \mathfrak{H}_g := \{Z \in M_{g\times g}(\C) \mid Z\ \mathrm{is}\ \mathrm{symmetric}\ \mathrm{and}\ \mathfrak{Im}(Z)\ \mathrm{is}\ \mathrm{positive}\ \mathrm{definite}\}.\]
\end{example}

\begin{example} Finally, in the case of a K3 surface $S$ the polarized Hodge structure on the second cohomology $H^2(S,\Z)$ has Hodge numbers $h^{2,0} = h^{0,2} = 1$ and $h^{1,1} = 20$.

Let $\sigma \in H^{2,0}$ be any non-zero element. We claim that $\sigma$ uniquely determines the subspaces $H^{1,1}$ and $H^{0,2}$. The subspace $H^{0,2}$ is easy, it is spanned by the complex conjugate $\overline{\sigma}$. The remaining subspace $H^{1,1}$ is then uniquely defined by the fact that it is orthogonal to $\sigma$ with respect to the bilinear form $Q$: to be precise, $H^{1,1}$ is the complexification of the real vector subspace of $H_{\Z} \otimes \mathbb{R} = \Lambda_{\mathrm{K3}}\otimes \mathbb{R}$ given as the orthogonal complement of the plane spanned by $\mathfrak{Re}(\sigma)$ and $\mathfrak{Im}(\sigma)$.

Thus to specify a point in $\calD$, it suffices to give $\sigma \in \Lambda_{\mathrm{K3}}\otimes \C$ that spans $H^{2,0}$ (in fact, since $\sigma$ is only defined up to non-zero scalar multiples, we only need the class of $\sigma$ in the projectivisation $\Proj(\Lambda_{\mathrm{K3}}\otimes \C)$). The Hodge-Riemann bilinear relations (2) and (3) become
\begin{enumerate}\addtocounter{enumi}{1}
\item $Q(\sigma,\sigma) = 0$,
\item $Q(\sigma,\overline{\sigma}) < 0$.
\end{enumerate}
Thus the local period domain for K3 surfaces is
\[ \calD :=\{[\sigma] \in \Proj(\Lambda_{\mathrm{K3}} \otimes \C) \mid Q(\sigma,\sigma) = 0, \ Q(\sigma,\overline{\sigma}) < 0 \}.\] 
It is a smooth $20$-dimensional quasi-projective variety.

\end{example}

\subsection{The Global Period Mapping} \label{sect:globmap}

Studying families over a polydisc $\Delta$ does not allow us to consider the situation where we have a family over a base that is not topologically trivial. If our base $B$ is a more general complex manifold, then the isomorphism $H^n(X_b,\Z) \cong H^n(X_{b'},\Z)$ for $b,b' \in B$ is not necessarily unique. This means that the period mapping is no longer well-defined. To compensate for this, we must quotient the period domain $\calD$ by the action of monodromy.

Let
\begin{equation} \label{Geq} \mathrm{Aut}(H_{\ZZ},Q) := \{g\colon H_{\Z} \to H_{\Z} \mid Q(g\xi, g\eta) = Q(\xi, \eta)\ \mathrm{for}\ \mathrm{all}\ \xi,\eta \in  H_{\Z}\}.\end{equation}
This group acts on $\calD$ in the obvious way. We have a monodromy representation
\[\varrho \colon \pi_1(B) \to \mathrm{Aut}(H_{\ZZ},Q).\]

Suppose that $\Gamma \subset \mathrm{Aut}(H_{\ZZ},Q)$ contains the image of $\varrho$. Then we have a well-defined map $\phi\colon B \to \Gamma \backslash \calD$. This is the \emph{global period mapping}. The quotient $ \Gamma \backslash \calD$ is called the \emph{period domain}.

We now return to our three examples.

\begin{example} In the case of an elliptic curve $E$, the group $\mathrm{Aut}(H_{\ZZ},Q)$ is the group of linear transformations $\Z^2 \to \Z^2$ that preserve the bilinear form $Q$. This is precisely the group $\mathrm{SL}(2,\Z)$. It acts on $\calD \cong \mathfrak{H}$ by
\[ \left[ \begin{array}{cc} a & b \\ c & d \end{array} \right] \cdot z = \frac{c + dz}{a + bz}. \]
Note that the negative identity matrix acts trivially, so we get an induced action of the modular group $\Gamma := \mathrm{PSL}(2,\Z)$ on $\calD$. The period domain for elliptic curves is therefore
\[ \Gamma \backslash \calD \cong \mathrm{PSL}(2,\Z) \backslash \mathfrak{H}.\]
This is the classical modular curve.
\end{example}

\begin{example} The case of a curve $C$ of genus $g \geq 1$ is similar. Here the group $\mathrm{Aut}(H_{\ZZ},Q)$ is the group of linear transformations $\Z^{2g} \to \Z^{2g}$ that preserve the bilinear form $Q$, which is precisely the group $\mathrm{Sp}(2g,\Z)$. It acts on $\calD \cong \mathfrak{H}_g$ by
\[ \left[ \begin{array}{cc} A & B \\ C & D \end{array} \right] \cdot Z = (C + DZ)(A + BZ)^{-1}, \]
where $A$, $B$, $C$, $D$ are $g \times g$ matrices. Note that the negative identity matrix $-I_{2g}$ acts trivially, so we get an induced action of the group $\Gamma_g := \mathrm{Sp}(2g,\Z) / \{\pm I_{2g}\}$ on $\calD$. The period domain for curves of genus $g \geq 1$ is therefore
\[ \Gamma \backslash \calD \cong \Gamma_g \backslash \mathfrak{H}_g.\]
It is a normal, quasi-projective variety.
\end{example}

\begin{example} Finally, in the case of a K3 surface $S$ the group $\mathrm{Aut}(H_{\ZZ},Q)$ of linear transformations $\Lambda_{\mathrm{K3}} \to \Lambda_{\mathrm{K3}}$ that preserve the bilinear form $Q$ does not act properly discontinuously on $\calD$. Thus the period domain $\mathrm{Aut}(H_{\ZZ},Q) \backslash \calD$ for K3 surfaces will not be a Hausdorff space. This can be rectified by restricting oneself to \emph{algebraic} K3 surfaces, which has the effect of shrinking the local period domain to a $19$-dimensional quasi-projective variety. Once this restriction has been performed, the group $\mathrm{Aut}(H_{\ZZ},Q)$ acts properly discontinuously and the resulting period domain is a quasi-projective variety with only finite quotient singularities. The interested reader may find more details in \cite[Chap. VIII]{bpv} or \cite{gmk3s}.
\end{example}

\subsection{Variations of Hodge Structure} \label{subsecvhs}

We can reverse engineer this theory to define abstract variations of polarized Hodge structure. Heuristically, a variation of polarized Hodge structure may be defined to be a map from a complex manifold into a period domain that satisfies the properties we have observed in Sects. \ref{sect:locmap} and \ref{sect:globmap}.

More rigorously, let $H_{\Z}$ be a finitely generated free abelian group equipped with a nondegenerate bilinear form $Q$. Let $\calD$ be a local period domain classifying Hodge structures of weight $n$ on $H_\CC = H_{\Z} \otimes_{\Z} \C$ polarized by the bilinear form $Q$, with given Hodge numbers $\{h^{p,q}\}$ (defined as in Definition \ref{Ddefn}). Define the group $\mathrm{Aut}(H_{\ZZ},Q)$ as in Eq. \eqref{Geq} and let $\Gamma \subset \mathrm{Aut}(H_{\ZZ},Q)$ be a subgroup. Finally, let $B$ be a complex manifold.

\begin{definition} \label{defn:PVHS} A map $\phi\colon B \to \Gamma \backslash \calD$ defines a \emph{polarized variation of Hodge structure of weight $n$} on $B$ if
\begin{enumerate}[(i)]
\item for every point $b \in B$, the map $\phi$ restricted to a small polydisc around $b$ lifts to a holomorphic map $\tilde{\phi}_b \colon \Delta \to \calD$ ($\phi$ is said to be \emph{locally liftable} and the maps $\tilde{\phi}_b$ are called \emph{local lifts}), and
\item the local lifts $\tilde{\phi}_b$ around any point $b \in B$ satisfy Griffiths transversality.
\end{enumerate}
\end{definition}

\begin{remark} We note that the Griffiths transversality condition (ii) is, in general, non-trivial. The exception occurs when $\calD$ is a \emph{Hermitian symmetric domain} \cite[Rem. 4.4.8]{pmpd}. This happens in two cases \cite[Exs. 4.3.5, 4.3.6]{pmpd}: when the weight $n = 1$ and $\calD$ is the Siegel upper half-space $\mathfrak{H}_g$ (known as \emph{Abelian variety type}), and when $n = 2$ and $\calD$ is a Type IV domain (known as \emph{K3 surface type}, as they arise in the study of variations of Hodge structures arising from families of algebraic K3 surfaces).
\end{remark}

In more generality, let $B$ be a complex manifold and let $\calE_{\Z}$ be a locally constant system of finitely generated free $\Z$-modules on $B$. Define $\calE := \calE_{\Z} \otimes \calO_B$. Then $\calE$ is a complex vector bundle which carries a natural flat connection $\nabla\colon \calE \to \calE \otimes \Omega_B^1$ (the \emph{Gauss-Manin connection}, see \cite[Sect. 9.2]{htcagI}),  induced by $\partial \colon \calO_B \to \Omega_B^1$. Let $\{\calF^p\}$ be a filtration of $\calE$ by holomorphic sub-bundles.

\begin{definition} \label{defn:VHS} The data $(\calE_\ZZ,\calF)$ defines a \emph{variation of Hodge structure of weight $n$} on $B$ if
\begin{enumerate}[(i)]
\item $\{\calF^p\}$ induces Hodge structures of weight $n$ on the fibres of $\calE$, and
\item if $s$ is a section of $\calF^p$ and $\zeta$ is a vector field of type $(1,0)$, then $\nabla_{\zeta}s$ is a section of $\calF^{p-1}$ (this is a reformulation of Griffiths transversality).
\end{enumerate}

Furthermore, if $\calE_{\Z}$ carries a nondegenerate bilinear form $Q \colon \calE_{\Z} \times \calE_{\Z} \to \Z$, we have a \emph{polarized variation of Hodge structure of weight $n$} if, additionally, the linear extension of $Q$ to $\calE$ satisfies
\begin{enumerate}[(i)]\addtocounter{enumi}{2}
\item $Q$ defines a polarized Hodge structure on the fibres of $\calE$, and
\item $Q$ is flat with respect to $\nabla$, i.e.
\[\partial Q(s,s') = Q(\nabla s, s') + Q(s, \nabla s')\]
for any sections $s$, $s'$ of $\calE$.
\end{enumerate}
\end{definition}

Then we find:

\begin{lemma} \textup{\cite[Lemma 4.5.3]{pmpd}} The definitions of polarized variation of Hodge structure of weight $n$ from Definitions \ref{defn:PVHS} and \ref{defn:VHS} agree.
\end{lemma}

Finally, we have:

\begin{definition} The \emph{Hodge bundles} $\calE^{p,n-p}$ associated to a variation of Hodge structure $(\calE_\ZZ,\calF)$ of weight $n$ are defined by $\calE^{p,n-p} = \calF^p/\calF^{p+1}$. 
\end{definition}

As one might expect, there is a $C^{\infty}$ (but \emph{not} holomorphic) decomposition 
\[\calE = \bigoplus_{p+q=n} \calE^{p,q}, \qquad \calE^{p,q} = \overline{\calE^{q,p}}.\]

\section{Mixed Hodge Structures}\label{sect:mhs}

The next section of these notes will be concerned with what happens to the Hodge structures on the cohomology of a K\"{a}hler manifold as that manifold degenerates to a singular variety. This provides the main motivation for the development of \emph{limiting mixed Hodge structures}. In this section we will give an overview of some of this theory.

The basic setup consists of a proper holomorphic map $f\colon \calX\ra\Delta$ onto the unit disc $\Delta \subset \CC$, that is smooth away from the fibre over zero. Assume that $X_b$, the fibre over $b\in\Delta$, is K\"ahler for each $b\neq 0$. Then for any $n$, the cohomology $H^n(X_b,\CC)$ carries a natural pure Hodge structure of weight $n$. This gives rise to a variation of Hodge structures, as discussed in the previous section.

Under this setup, in 1970 Griffiths \cite{Gr70} conjectured that monodromy around $0 \in \Delta$ should give a filtration $W$, called a \emph{weight filtration}, on the cohomology $H_{\QQ}=H^n(X_b,\QQ)$ of a nearby fibre and, moreover, that there should exist a suitable Hodge filtration $F_\llim$ on $H_{\QQ}$ so that the triple $(H_{\QQ},W,F)$ defines a so-called \emph{mixed Hodge structure}, i.e. $F_\llim$ induces a pure Hodge structure of weight $k+n$ on the $k$th $W$-graded piece $\Gr_k^W H_{\QQ}$ .

The first progress on this conjecture came in 1971, when Deligne \cite{DelTH2} proved that, for every $n$, the $n$th cohomology of a smooth variety $U$ over $\CC$ carries a natural and functorial mixed Hodge structure. Using this, in 1973 a seminal work by Schmid \cite{Sc73} defined the filtration $F_\llim$, verified Griffiths' prediction and studied abstract aspects of limiting mixed Hodge structures.

The next year, in 1974, Deligne \cite{DelTH3} extended his previous result to arbitrary varieties over $\CC$ (including singular ones) and Malgrange \cite{Ma74} studied the asymptotic behaviour of flat sections of the Gauss-Manin connection near the singular points of a regular function. 

Based on these results, in 1976, Steenbrink \cite{St75} was able to give a comprehensive account of the case where the central fibre $X_0$ is a semi-stable normal crossing divisor and the map $f$ is projective. He then proved that the sequence
\[H^n(X_0,\CC)\stackrel{r^*}{\lra} H^n(X_t,\CC)\stackrel{\log T}{\lra} H^n(X_t,\CC)\]
is an exact sequence of mixed Hodge structures, where $r\colon X_t\ra X_0$ is a retraction map and $T\colon  H^n(X_t,\CC)\ra H^n(X_t,\CC)$ is the monodromy operator given by parallel transport around $0$. This \emph{local invariant cycle theorem} had been conjectured by Griffiths \cite{Gr70}. Steenbrink's results were later extended to the K\"ahler case in 1977, by Clemens \cite{Cl77}.

Finally (for us at least!), in 1980, Var\v{c}enko \cite{Va80} showed that near a singularity of a regular function, the asymptotics of relative differential forms determine a mixed Hodge structure.
 
Our aim is to give an overview of some of this theory. The exposition in this section will follow closely the comprehensive book by Peters and Steenbrink \cite{PS08}, with additional examples to illustrate the main ideas.

\subsection{Mixed Hodge Structures on Smooth Varieties}\label{mhssmooth}

Before we can start studying degenerations of Hodge structures, we first need to introduce some definitions.

\label{sec:2}
\begin{definition}
\begin{enumerate}
\item A \emph{mixed Hodge structure} $(H_\ZZ,W,F)$ consists of a $\ZZ$-mod\-ule $H_\ZZ$ together with an increasing filtration $W$
\[\cdots\subset W_0\subset W_1\subset W_2\subset \cdots\] 
of $H_\QQ := H_\ZZ\otimes_\ZZ\QQ$ and a decreasing filtration $F$
\[H_{\CC} =F^0\supset F^{1}\supset F^2\supset\cdots\] 
of $H_{\CC} :=H_\ZZ\otimes_\ZZ\CC$ such that $F$ defines a (pure) Hodge structure of weight $k$ on the graded piece $\Gr^W_k H_\QQ=W^{k}H_\QQ/W^{k+1}H_\QQ$.
\item A mixed Hodge structure $(H_\ZZ,W,F)$ is \emph{graded-polarizable} if each graded piece $\Gr^W_k H_\QQ$ is polarizable.
\item The \emph{Hodge numbers} of a mixed Hodge structure $(H_\ZZ,W,F)$ are defined to be
\[h^{p,q}(H)=\dim_\CC \Gr^p_F\Gr^W_{p+q}H_\CC\]
\item A \emph{mixed Hodge structure of weight $n$}, for $n \in \ZZ$, is a triple $(H_\ZZ,W,F)$ such that $(H_\ZZ,W[-n],F)$ is a mixed Hodge structure, where $W[-n]$ denotes the shifted filtration defined by $W[-n]_{\bullet} := W_{\bullet - n}$. In particular, note that the Hodge numbers of a mixed Hodge structure of weight $n$ are given by $h^{p,q}(H)=\dim_\CC \Gr^p_F\Gr^W_{p+q-n}H_\CC$
\end{enumerate}
\end{definition}

One of the simplest examples of a mixed Hodge structure is induced by the Hodge filtration on the cohomology of a compact K\"{a}hler manifold.

\begin{example} 
\label{example-MHS1}
Let $X$ be a compact K\"ahler manifold and set 
\begin{eqnarray*}
H_\ZZ &=& \bigoplus_i H^i(X,\ZZ)/\mathrm{torsion},\\
W_kH_\QQ &=& \bigoplus_{i\le k} H^i(X,\QQ),\\
F^pH_\CC &=& \bigoplus_{i} F^pH^i(X,\CC),
\end{eqnarray*}
where $F^pH^i(X,\ZZ)$ denotes the usual Hodge filtration of $H^i(X,\CC)$. Then $(H_\ZZ,W,F)$ is a mixed Hodge structure.
\end{example}

The first main result about mixed Hodge structures is the following, originally due to Deligne \cite{DelTH2}. We will discuss this result further in Ex. \ref{omegalogDex}, a more detailed discussion may also be found in \cite[Ch. 4]{PS08}.

\begin{theorem}\textup{\cite{DelTH2}}
\label{deligne-on-smooth}
Let $U$ be a smooth variety \textup{(}that is not necessarily compact\textup{)}. Then $H_\ZZ=H^k(U,\ZZ)$ carries a natural mixed Hodge structure, which is functorial with respect to maps of algebraic manifolds $U \to V$. Moreover, $\Gr_i^W H_\QQ=0$ unless $k\le i\le \min(\dim U,2k)$.
\end{theorem}

To illustrate the usefulness of such mixed Hodge structures, consider the situation where $X$ is a smooth projective variety and $Y\subset X$ is a closed subvariety. Let $U=X\setminus Y$ denote the open complement of $Y$ in $X$. We obtain a long exact sequence of relative cohomology groups
\begin{equation}
\label{relative-coho-seq}
\cdots\longrightarrow H^i(X,U)\lra H^i(X)\lra H^i(U)\lra H^{i+1}(X,U)\lra \cdots
\end{equation}

An important feature of this sequence is that it upgrades to an exact sequence of mixed Hodge structures, where the Hodge structure on $H^i(X)$ is the usual pure Hodge structure on the cohomology of a smooth projective variety and $H^i(U)$ and $H^{i}(X,U)$ both carry mixed Hodge structures. 

We will use this to study mixed Hodge structures in two specific examples. Both will feature elliptic curves, as their cohomologies carry fairly simple Hodge structures (Ex. \ref{example-ell-curve}).

\begin{example}[Punctured elliptic curves] \label{puncturedellipticex}
It follows from Torelli's theorem (see, for instance, \cite[Thm. 3.5.2]{pmpd}) that an elliptic curve $E$ is determined uniquely up to isomorphism by the Hodge structure on $H^1(E,\ZZ)\cong \ZZ^2$. This statement remains true if we consider the non-compact variety $E\setminus \{p\}$, where $p$ is a point on $E$. Indeed, puncturing once does not change $H^1(E,\ZZ)$ (instead it kills $H^2(E,\ZZ)$), so 
\[H^1(E,\ZZ) = H^1(E\setminus\{p\},\ZZ).\] 

Things become more interesting when we puncture the elliptic curve more than once. Let $U := E\setminus\{p_1,...,p_k\}$, for $k>0$ and $p_j$ distinct. Then
\[H^1(U,\QQ)\cong \QQ^{k+1}\]
When $k$ is even, we see that $H^1(U,\QQ)$ has odd rank. It therefore cannot support a pure Hodge structure of weight one, as any lattice supporting such a Hodge structure must necessarily have even rank $h^{1,0}+h^{0,1} = 2h^{1,0}$.

Instead, there is a mixed Hodge structure on $H^1(U,\QQ)$ (computed using the method of Ex. \ref{omegalogDex}), with
\[\Gr_0^WH^1(U,\QQ)=0,\qquad \Gr_1^WH^1(U,\QQ)\cong\QQ^2,\qquad \Gr_2^WH^1(U,\QQ)\cong\QQ^{k-1}.\]
 
Now consider the long exact sequence of relative cohomology \eqref{relative-coho-seq}.  The maps in this sequence respect the weight filtration, so we can decompose it into graded pieces. Noting that 
\[H^k(E,U;\QQ)=H^k_{\{p_1,\cdots,p_k\}}(E,\QQ)\cong\left\{\begin{array}{ll} \QQ^{k} & k=2\\ 0&k\neq 2,\end{array}\right. \]
we see that decomposition into graded pieces gives the three exact sequences
\[ \begin{array}{lcc}
\Gr^W_0\colon&\qquad& 0\lra H^0(E)\stackrel{\sim}{\longrightarrow} H^0(U) \lra 0 \\[2mm]
\Gr^W_1\colon&& 0 \lra H^1(E) \stackrel{\alpha}{\longrightarrow} \Gr^W_1H^1(U)\lra  \Gr^W_1H^2(E,U) \lra 0 \\[2mm]
\Gr^W_2\colon&& 0\lra \Gr^W_2H^1(U)\lra \Gr^W_2H^{2}(E,U) \ra H^2(E)  \lra 0 \\[2mm]
\end{array}\]

As $\Gr_1^WH^1(U,\QQ)\cong\QQ^2$, we see that $\Gr^W_1H^2(E,U)=0$ and $\alpha$ must be an isomorphism. So the pure Hodge structure on $\Gr_1^WH^1(U,\QQ)\cong\QQ^2$ is precisely the same as the one on the cohomology $H^1(E,\QQ)$ of the unpunctured curve; intuitively, we see that the mixed Hodge structure of the punctured curve still knows which elliptic curve it is.  

We also see that the third sequence above is isomorphic to an exact sequence of the form
\[0\lra \QQ^{k-1} \lra \QQ^{k} \lra \QQ  \ra 0.\]
From this we can see that the mixed Hodge structure of $H^2(E,U)$ must be concentrated in $\Gr_2^W$.
\end{example}

\begin{example}[Affine complement of an elliptic curve] 
Let $E$ be an elliptic curve embedded in $\PP^2$ and let $U:=\PP^2\setminus E$. 
It is a fun application of the techniques in \cite{RSTZ14} to show that
$$H_k(U,\ZZ)\cong\left\{\begin{array}{ll} 0&k\ge 3\\ \ZZ^{2} & k=2\\ \ZZ/3 &k= 1\\ \ZZ&k=0.\end{array}\right.$$
In particular $\dim H_2(U,\QQ)=2$ and hence, by Poincar\'e duality, $\dim H^2_c(U,\QQ)=2$, where $H^\bullet_c$ denotes cohomology with compact support.

Since $H^1(\PP^2)=0$, the long exact sequence of cohomology with compact support reads
\[ 0 \lra H^1(E) \stackrel{\alpha}{\longrightarrow} H_c^2(U)\lra H^2(\PP^2)\lra \cdots\]
and, since we know the ranks of the first few terms, we see that $\alpha$ must be an isomorphism. 

This sequence also upgrades to a sequence of mixed Hodge structures, giving
\[\Gr_1^W H_c^2(U) = H^1(E).\]
Moreover, Poincar\'e duality is compatible with mixed Hodge structures and in this case yields
\[\Gr_3^W H^2_c(U) = H^1(E)^{\vee},\]
where $H^1(E)^{\vee}$ denotes the dual of the Hodge structure $H^1(E)$. The upshot is that the mixed Hodge structure on the cohomology of $U$ knows about the elliptic curve $E$ that was removed from $\PP^2$ to obtain $U$.
\end{example} 

\subsection{Mixed Hodge Structure via a Bifiltered Complex} \label{sect:mhsbc}

A common way to construct mixed Hodge structures is via bifiltered complexes. As in Sect. \ref{specseqsec}, let 
\[\cdots \stackrel{d}{\longrightarrow} A^0\stackrel{d}{\longrightarrow} A^1 \stackrel{d}{\longrightarrow} A^2 \stackrel{d}{\longrightarrow} \cdots\]
be a complex of sheaves on a space $X$. 
%Recall that the \emph{hypercohomology} of the complex of sheaves $A^\bullet$ is defined by choosing an acyclic resolution of $A^\bullet$ by a double complex $I^{\bullet,\bullet}$, i.e. a diagram with exact rows and columns
%\begin{center}
%\centerline{
%\xymatrix@C=30pt
%{
% & \vdots&\vdots&\vdots\\
%\cdots\ar[r] & I^{0,1} \ar_d[r]\ar_{d'}[u]& I^{1,1} \ar_d[r]\ar_{d'}[u]& I^{2,1} \ar_d[r]\ar_{d'}[u]&\cdots\\
%\cdots\ar[r] & I^{0,0} \ar_d[r]\ar_{d'}[u]& I^{1,0} \ar_d[r]\ar_{d'}[u]& I^{2,0} \ar_d[r]\ar_{d'}[u]&\cdots\\
%\cdots\ar[r] & A^0 \ar_d[r]\ar_{d'}[u]& A^1 \ar_d[r]\ar_{d'}[u]& A^2 \ar_d[r]\ar_{d'}[u]&\cdots\\
% & 0 \ar[u] & 0 \ar[u] & 0 \ar[u]
%}}
%\end{center}
%and $H^k(X,I^{i,j})=0$ for $k\ge 0$. Such a resolution always exists. 
%Let $I^k=\bigoplus_{i+j=k} I^{i,j}$ denote the total complex with differential $\delta=d+(-1)^jd'$, then
%\[\HH^k(X,A^\bullet):=H^k_\delta \Gamma(X,I^\bullet)\]
%defines the $k$th hypercohomology of $A^\bullet$, which is independent of the choice of $I^{\bullet,\bullet}$.
Assume that we have an increasing filtration $W$ and a decreasing filtration $F$ on $A^\bullet$ turning it into a bifiltered complex, i.e. $d$ preserves the filtrations so that $d\colon W_iA^k\ra W_iA^{k+1}$ and $d\colon F^iA^k\ra F^iA^{k+1}$. Then $W$ and $F$ induce filtrations (also called $W$ and $F$) on the hypercohomologies $\HH^n(X,A^{\bullet})$ in the way described in Sect. \ref{specseqsec}.
 
In order for $(\HH^n(X,A^\bullet),W,F)$ to give a mixed Hodge structure (of weight $m$), we also need to have a rational version $(A_\QQ^\bullet,W_\QQ)$ of $(A^\bullet,W)$ and a $\ZZ$-version $A_\ZZ^\bullet$ of $A^\bullet$. Furthermore we require that, for any $k$, the spectral sequence of the filtered complex $(\Gr_k^WA^\bullet,F)$ degenerates at $E_1$ and induces a Hodge structure of weight $k+m$ on $\Gr_k^W\HH^n(X,A^\bullet)$ (see Sect. \ref{specseqsec}).

If this is satisfied, one calls $(A^{\bullet},W,F)$ a \emph{cohomological mixed Hodge complex}, see \cite[Sect. 8]{DelTH3}; most known mixed Hodge structures are constructed this way. An important example of this is the following.

%Deligne \cite[Sch. 8.1.9]{DelTH3} has shown that our assumptions imply that the spectral sequence of the filtered complex $(A^\bullet, W)$ always degenerates at $E_2$.

\begin{example} \label{omegalogDex}
\label{example-MHS2}
Let $U$ be a smooth variety and $X\supseteq U$ a compactification, such that $X\setminus U=D$ is a normal crossing divisor in $X$. Let $\Omega^\bullet_X(\log D)$ denote the differential forms on $X$ with at most logarithmic poles along $D$, i.e. $\Omega^r_X(\log D)$ is generated by forms of the shape
$\frac{dz_1}{z_1}\wedge\cdots\wedge\frac{dz_j}{z_j}\wedge \alpha$, where $z_i$ is a local equation of a component of $D$, $j\le r$ and $\alpha\in\Omega^{r-i}_X$. Then $A^\bullet=\Omega_X^\bullet(\log D)$ is a complex of sheaves under the de Rham differential. Furthermore, one can show \cite[Thm. 4.2]{PS08} that
\[\HH^n(X,\Omega_X^\bullet(\log D)) = H^n(U,\CC).\]

There are two filtrations on $\Omega^\bullet_X(\log D)$, given by
\begin{eqnarray*}
W_m\Omega^i_X(\log D) &=& \left\{\begin{array}{ll} \Omega^i_X(\log D)&\hbox{if }i\le m\\ \Omega_X^{i-m}\wedge\Omega^m_X(\log D)\quad&\hbox{if }0\le m\le i\\0&\hbox{if }m<0\end{array}\right.,\\
F^p\Omega^i_X(\log D) &=& \left\{\begin{array}{ll} \Omega^i_X(\log D)\quad&\hbox{if }p\le i\\ 0&\hbox{otherwise.}\end{array}\right.
\end{eqnarray*}
The main result of \cite{DelTH2} states that this defines a cohomological mixed Hodge complex $(\Omega_X^\bullet(\log D)),W,F)$. This cohomological mixed Hodge complex gives rise to the natural mixed Hodge structure on $H^n(U,\C)$ given by Thm. \ref{omegalogDex}.
\end{example}

\begin{example} 
\label{example-MHS3}
As a more concrete version of Ex. \ref{omegalogDex}, let us compute the mixed Hodge structure of $U = \PP^1\setminus\{p_1,\ldots,p_k\}$, with $p_i$ distinct points and $k\ge 1$. The bifiltered complex $A^\bullet = \Omega_{\Proj^1}^\bullet(\log \{p_1,\ldots,p_k\})$ is given by
\[\shO_{\PP^1}\stackrel{d}{\lra}\Omega^1_{\PP^1}(\log\{p_1,\ldots,p_k\}).\]
As both terms have no higher cohomology, we can just take the total complex $I^\bullet$ of the resolution to coincide with $A^\bullet$.

Taking global sections, the complex becomes
\[\CC\cong\Gamma(\PP^1,\shO_{\PP^1})\stackrel{d}{\lra}\Gamma(\PP^1,\Omega^1_{\PP^1}(\log\{p_1,\ldots,p_k\}))\]
and the differential becomes trivial. The right hand side is isomorphic to $\CC^{k-1}$ (by taking residues and finding that they sum to zero).
Computing $\HH^n(\Proj^1,A^\bullet)$, we therefore conclude that $H^2(U,\CC)=0$ and
\begin{eqnarray*}
h^{p,q}H^1(U,\CC) =& \dim \Gr_F^p\Gr^W_{p+q}\HH^1(\Proj^1,A^\bullet) &= \left\{\begin{array}{ll} k-1 &\quad p=q=1\\ 0&\quad\hbox{otherwise,}\end{array}\right.\\
h^{p,q}H^0(U,\CC) =& \dim \Gr_F^p\Gr^W_{p+q}\HH^0(\Proj^1,A^\bullet) &= \left\{\begin{array}{ll} 1 &\quad p=q=0\\ 0&\quad\hbox{otherwise.}\end{array}\right.
\end{eqnarray*}
The main difference between this and Ex. \ref{example-MHS1} is that here $H^1(U,\CC)$ is concentrated in weight two rather than weight one. In other words, the Hodge structure is still pure, just in a different weight. 

More generally, several different weights may contribute to the cohomology $H^1(U,\CC)$, as happened in Ex. \ref{puncturedellipticex}. However, Thm. \ref{omegalogDex} shows that only weights greater than or equal to $n$ can contribute to the mixed Hodge structure of $H^n(U,\CC)$ for a smooth variety $U$. To get contributions from weights below the cohomology degree, one needs to look at Hodge structures of singular varieties; see Ex.~\ref{example-MHS4}.
\end{example}

\subsection{Extending Vector Bundles with Connection and Limiting Mixed Hodge Structures}
\label{schmidswork}

We can also define a mixed Hodge structure on the central fibre of a degeneration. The mixed Hodge structure obtained in this way is called the \emph{limiting mixed Hodge structure}. For a more detailed discussion, we refer the interested reader to \cite[Ch. 11]{PS08}.

We begin with some general results about vector bundles over punctured discs. The reader should keep in mind the special case of vector bundles $\calE$ arising from variations of Hodge structure $(\calE_{\Z},\calF)$, as we will specialize to this case shortly. 

Begin by letting $\Delta \subset \CC$ denote the unit disk and let $\calE$ be a vector bundle on $\Delta$ with connection $\nabla$ on $\calE|_{\Delta\setminus \{0\}}$. 

\begin{definition} \label{logpolesdefn}
$\nabla$ is said to have \emph{logarithmic poles along $\{0\}$} if $\nabla$ extends to a \emph{logarithmic connection}
\[\nabla\colon \calE\lra\calE\otimes_{\shO_\Delta}\Omega^1_\Delta(\log \{0\})\]
(still satisfying the Leibniz rule $\nabla(f\cdot s) = f \nabla(s) + df \otimes s$).
\end{definition}

Now define a map $R$, the \emph{Poincar\'{e} residue map}, by
\begin{eqnarray*}
R\colon \Omega^1_\Delta(\log\{0\})&\lra&\shO_{\{0\}}\\
\omega &\longmapsto& f|_{\{0\}},
\end{eqnarray*}
where $\omega=f\wedge\frac{dz}z + \eta$, for $z$ a coordinate on $\Delta$ vanishing at the origin and some $f\in\shO_\Delta$ and $\eta\in\Omega^1_\Delta$.

It is easy to check \cite[Sect. 11.1]{PS08} that the map $(1 \otimes R) \circ \nabla \colon \calE \to \calE \otimes \calO_{\{0\}}$ induces an $\calO_{\{0\}}$-linear endomorphism
\[\res(\nabla) \in\End(\calE \otimes\calO_{\{0\}}).\]
This endomorphism is called the \emph{residue} of $\nabla$ at $\{0\}$.

Now suppose that $\calE$ is a vector bundle with connection $\nabla$ on $\Delta\setminus\{0\}$. The next theorem shows that we can extend $\calE$ to a vector bundle with logarithmic connection over $\Delta$,  whose residue at $\{0\}$ has a particular form.

\begin{theorem} \textup{\cite[Prop. 11.3]{PS08}}
Let $\tau\colon \CC/\ZZ\ra\CC$ be a section of the projection $\CC\ra\CC/\ZZ$. 
%We define the associated logarithm as the univalent function on $\CC^*$ given by
%\[\log_\tau(z)=2\pi i\tau\Big(\underset{\hbox{\scriptsize maps }\CC\setminus\{0\}\ra\CC/\ZZ}{\underbrace{\frac{\log z}{2\pi i}}}\Big).\]
Then there is a unique extension $(\calE_\tau,\nabla_\tau)$ of $(\calE,\nabla)$ to $\Delta$, such that $\nabla_\tau$ has logarithmic poles along $\{0\}$ and $\res(\nabla)$ has eigenvalues in the image of $\tau$.
\end{theorem}

\begin{definition}
The \emph{canonical extension} is the unique extension where $\res(\nabla)$ has eigenvalues in $[0,1)$, i.e. the extension obtained by choosing $\tau\colon \CC/\ZZ\ra\CC$ to have image in $\{z\mid 0\le\mathfrak{Re}(z)<1\}$.
\end{definition}

%Finally, recall that an endomorphism $T$ is called \emph{unipotent} if $(T-\id)^k=0$ for some $k$.

Next we specialize this discussion to the setting where the vector bundle $\calE$ arises from a variation of Hodge structure over $\Delta \setminus \{0\}$; a more detailed discussion of these concepts may be found in \cite[Sect. 11.2.1]{PS08}.

So let $(\calE_\ZZ,\calF)$ be a polarized variation of Hodge structure on $\Delta\setminus\{0\}$ (see Sect. \ref{subsecvhs}). Let $T\colon \calE_\ZZ \to \calE_\ZZ$ denote the monodromy automorphism defined by parallel transport along a counterclockwise loop about $0 \in \Delta$. Then $T$ induces an automorphism of $\calE := \calE_{\Z} \otimes \calO_{\Delta\setminus\{0\}}$. Let $T=T_sT_u$ be the Jordan decomposition of this automorphism into semi-simple and unipotent parts, i.e.
\[T=\begin{pmatrix}\lambda & 1\\ &\lambda\end{pmatrix}= \underset{T_s}{\begin{pmatrix}\lambda & \\ & \lambda\end{pmatrix}} \underset{T_u}{\begin{pmatrix}1 & \frac1\lambda\\ & 1\end{pmatrix}}.\] The following theorem gives an important property of $T$.

\begin{theorem}[Monodromy theorem] \label{monodromythm} \textup{\cite[Lemma 4.5, Thm. 6.1]{Sc73}} The monodromy operator $T$ is quasi-unipotent. More precisely, if $m:=\max\{p-q\mid\calE^{p,q}\neq 0\}$, then $T_s$ has finite order and 
\[(T_u-\id)^{m+1}=0.\]
\end{theorem}

Now let  $\tilde{\calE}$ denote the canonical extension of $\calE$ to $\Delta$ and let $\calE_\infty=\tilde{\calE}|_{\{0\}}$ denote its central fibre. There is an integral structure $(\calE_\infty)_\ZZ$ on $\calE_\infty$, induced from $\calE_\ZZ$ by the canonical extension. 

Moreover, by \cite[Prop 11.2]{PS08}, $T$ extends to an automorphism of $\tilde{\calE}$, whose restriction to $\calE_{\infty}$ is given by $\exp(-2\pi i \res(\nabla))$. Let
\[N := \log T_u =\sum_{i\ge 1} (-1)^{i+1}\frac1i(T_u-\id)^i\]
denote the logarithm of the unipotent part $T_u$ of $T$; this is well-defined as $(T_u - \id)$ is nilpotent. We can use this $N$ to define a mixed Hodge structure on $\calE_{\infty}$ as follows:

\begin{theorem}[Schmid's limiting mixed Hodge structure] \textup{\cite[Thm. 6.16]{Sc73}} \label{limmhsthm}
The subbundles given by the Hodge filtration
\[\shF^n\subseteq\shF^{n-1}\subseteq\cdots\subseteq\shF^1\subseteq\shF^0=\calE\]
extend to holomorphic subbundles
\[\tilde{\shF}^n\subseteq\tilde{\shF}^{n-1}\subseteq\cdots\subseteq\tilde{\shF}^1\subseteq\tilde{\shF}^0=\tilde{\calE}\]
over $\Delta$. Define a decreasing filtration  $F^{\bullet}_{\operatorname{lim}}$ of $\calE_\infty$ by $F^k_{\operatorname{lim}}=\tilde{\shF}^k|_{\{0\}}$ and let $W_{\bullet}$ denote the weight filtration on $\calE_\infty$ induced by the restriction of $N$ to $\calE_\infty$, as described in the following subsection. Then
\[(\calE_\infty,W_\bullet,F^\bullet_{\operatorname{lim}})\]
defines a mixed Hodge structure, called the \emph{limiting mixed Hodge structure}.
\end{theorem}

\subsubsection{Weight Filtration from an Endomorphism}
\label{sec-weights-from-endo}

It remains to construct the weight filtration $W$ from Thm. \ref{limmhsthm}. Let $V$ be a vector space and let $N\in\End(V)$ be nilpotent, i.e. $N^{n+1}=0$ for some $n$. Then we have:

\begin{lemma} \textup{\cite[Lemma 11.9]{PS08}}
There is an increasing filtration 
\[W_0\subseteq W_1\subseteq\cdots\subseteq W_{2n}\]
of subspaces of $V$ uniquely determined by the properties
\begin{eqnarray*}
N(W_i)\subset W_{i-2}&\quad&\ \hbox{ for }i\ge 2\hbox{ and} \\
N^k\colon \Gr^W_{n+k}\lra \Gr^W_{n-k}&&\ \hbox{  is an isomorphism,}
\end{eqnarray*}
where we use the notation $\Gr^W_{i}=W_i/W_{i-1}$. We call $W$ the \emph{weight filtration centered at $n$}.
\end{lemma}
This is typically proved using an induction argument. Somewhat more hands-on, one can also prove it by splitting $V$ using an appropriate choice of basis; this is visualized in Fig. \ref{fig:weight-filt}, which is inspired by \cite[Sect. II.2.7]{Ku98}.

\begin{figure}[t]
\includegraphics[width=0.63\textwidth]{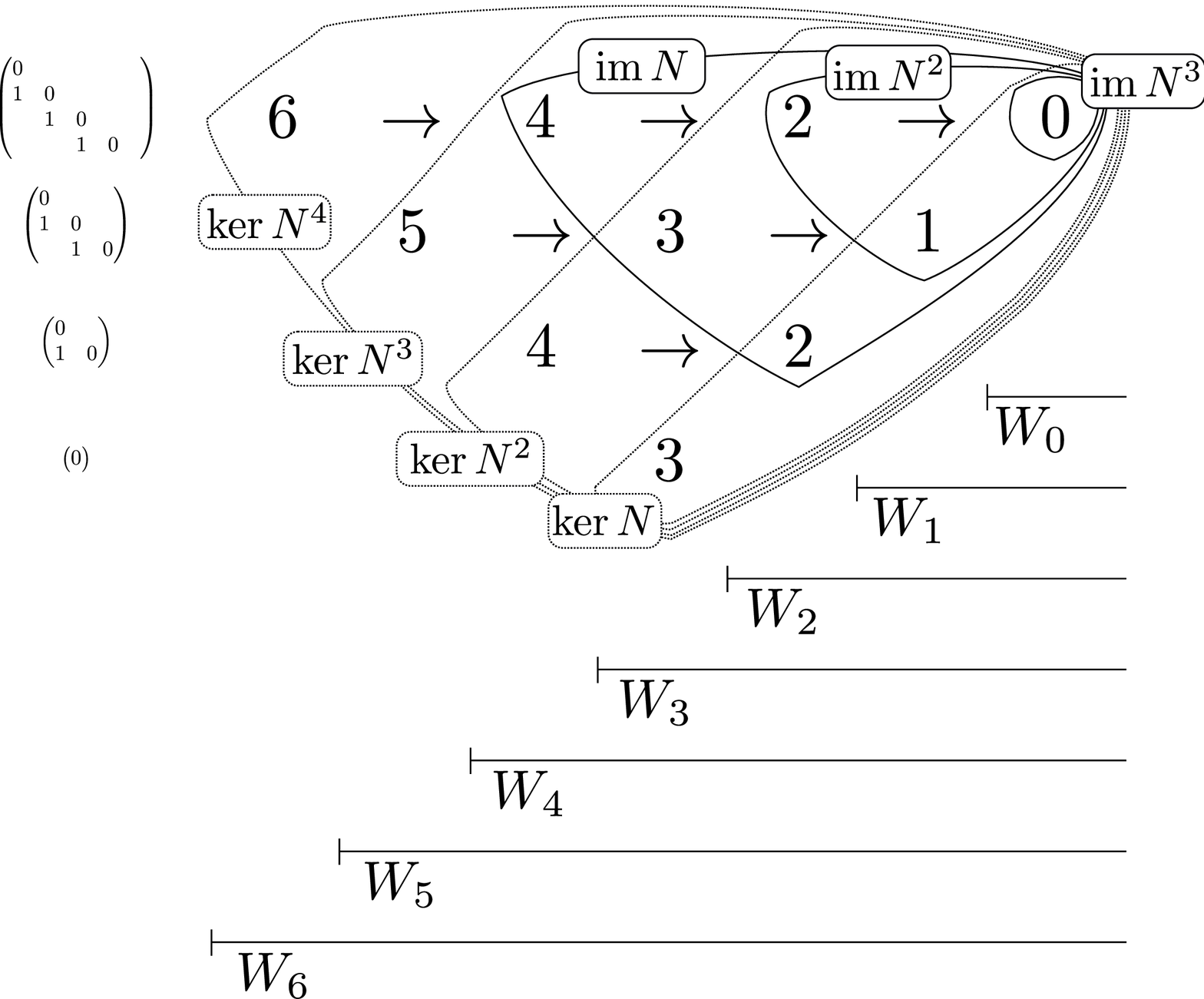}
\caption{A weight filtration centered at $3$ given by splitting a vector space $V$. It shows the schematic structure of the weight filtration in relation to kernels and images of powers of $N$. The horizontal arrows are $N$. The numbers indicate which graded piece of $W$ the part of $V$ lies in. The leftmost column shows a Jordan block for the row.} 
\label{fig:weight-filt}    
\end{figure}

One deduces the following for the weight filtration:
\begin{eqnarray*}
W_0&=&\im N^n\ \ \,(\cap \ker N),\\
W_1&=&\im N^{n-1}\cap \ker N,\\
W_2&=&\im N^{n-2}\cap \ker N\,+\,\im N^{n-1}(\cap \ker N^2),\\
W_3&=&\im N^{n-3}\cap \ker N\,+\,\im N^{n-2} \cap \ker N^2,\\
W_4&=&\im N^{n-4}\cap \ker N\,+\,\im N^{n-3} \cap \ker N^2\,+\,\im N^{n-2}(\cap \ker N^3),\\
&\vdots&\\
W_{2n-1}&=&\ker N^n,\\
W_{2n}&=&V
\end{eqnarray*}

\subsection{Limiting Mixed Hodge Structure of a Normal Crossing Degeneration}

Deligne \cite{DelTH3} extended the proof of the existence of a natural mixed Hodge structure on the cohomology of smooth varieties (Thm. \ref{deligne-on-smooth}) to work for arbitrary varieties. In particular, there is a mixed Hodge structure on the cohomology of a normal crossing divisor, that we will define in this section. We will then see that, for a degeneration of a smooth variety to a variety with normal crossings, Schmid's limiting mixed Hodge structure can be given concretely as the hypercohomology of a bifiltered complex; this parallels the method we used to compute the mixed Hodge structure of a smooth variety in Ex. \ref{omegalogDex}. Finally, we will see that these two mixed Hodge structures are related by several exact sequences.

\subsubsection{Mixed Hodge Structures of a Normal Crossing Divisor}\label{mhsncd}

We begin by defining Deligne's \cite{DelTH3} mixed Hodge structure on the cohomology of a normal crossing variety, following Steenbrink \cite[Sect. 3]{St75}. Let $Y$ be a simple normal crossing space with irreducible components $Y_1,\ldots,Y_N$. For $k \geq 0$, define the \emph{codimension $k$ stratum} of $Y$ to be
\[\tilde{Y}^{(k)} = \coprod_{i_0 < \cdots < i_{k}} Y_{i_0} \cap \cdots \cap Y_{i_{k}}\]
and let $a_k \colon \tilde{Y}^{(k)} \ra Y$ denote the natural map. Note that $a_1 \colon \tilde{Y}^{(0)} \to Y$ is the usual normalization map. Let $\delta_j\colon \tilde{Y}^{(k)} \to \tilde{Y}^{(k-1)}$ ($0 \leq j \leq k$) denote the map whose components are the inclusions
\[Y_{i_0} \cap \cdots \cap Y_{i_{k}} \lra Y_{i_0} \cap \cdots \cap \hat{Y}_{i_j} \cap \cdots \cap Y_{i_{k}},\]
where $\hat{Y}_{i_j}$ denotes the omission of $Y_{i_j}$.

Define a double complex
\[B^{p,q}= (a_{q})_*\Omega^p_{\tilde{Y}^{(q)}}\]
with $d\colon B^{p,q}\ra B^{p+1,q}$ induced by the de Rham differential $\partial\colon \Omega_{\tilde{Y}^{(q)}}^p \to \Omega_{\tilde{Y}^{(q)}}^{p+1}$, and $\delta\colon B^{p,q}\ra B^{p,q+1}$ given by the alternating restriction map
\[\delta :=\sum_{j=0}^{q+1} (-1)^{q+j+1}\delta_j^*.\] 

Let $B^{\bullet}$ denote the associated total complex. Then the discussion in \cite[Ex. 3.5]{St75} shows that $H^n(Y,\CC) \cong \HH^n(Y,B^{\bullet})$.

Next consider the two filtrations $W_\bullet$ and $F^\bullet$ defined on $B^{\bullet,\bullet}$ via
\[ W_k B^{p,q}=\left\{\begin{array}{lcl} B^{p,q} &\quad&\hbox{for } q\le -k\\ 0 &&\hbox{otherwise},\end{array}\right.
\qquad\qquad
F^k B^{p,q}=\left\{\begin{array}{lcl} B^{p,q} &\quad&\hbox{for } p\ge k\\ 0 &&\hbox{otherwise}.\end{array}\right.\]
This induces a cohomological mixed Hodge complex $(B^{\bullet},W,F)$ on $B^{\bullet}$, which gives a mixed Hodge structure of weight $n$ on the cohomology $H^n(Y,\CC)$. Indeed, if we denote the irreducible components of $\tilde{Y}^{(q)}$ by $\tilde{Y}^{(q)}_i$, then
\begin{equation}\Gr_{-q}^WB^{p+q}=B^{p,q}= \bigoplus_i(a_{q})_*\Omega^p_{\tilde{Y}^{(q)}_i},
\label{B-gradeds} \end{equation}
so $\Gr_{-q}^WB^{\bullet}$ is just a direct sum of de Rham complexes of smooth projective manifolds. Thus, by the discussion in Sect. \ref{specseqsec}, we have
\[\HH^n(Y,\Gr_{-q}^WB^\bullet) = \bigoplus_i \HH^n(Y, \Omega^{\bullet - q}_{\tilde{Y}^{(q)}_i}) = \bigoplus_i H^{n-q}(\tilde{Y}^{(q)}_i,\CC),\] 
and the natural pure Hodge structure on $H^{n-q}(\tilde{Y}^{(q)}_i,\CC)$ is induced by $F$. This is exactly what is required for the $W$-graded pieces of a mixed Hodge structure of weight $n$ (as well as for a cohomological mixed Hodge complex).

\begin{example} 
\label{example-MHS4}
Consider the subvariety $Y=V(xyz)$ of $\PP^2=\CC\Proj[x,y,z]$ given as the union of the three coordinate lines; this is a necklace of three copies of $\PP^1$. Our aim is to compute the mixed Hodge structure of $H^n(Y,\CC) = \HH^n(Y,B^{\bullet})$ (for $B^{\bullet}$ as above). 

In order to do this, consider the filtered complex $(B^\bullet, W_\bullet)$. The filtration $W_{\bullet}$ is increasing, but we can convert it into a decreasing filtration $W^{\bullet}$ by setting $W^i:=W_{-i}$. From this, by Thm.~\ref{thm:specseq}, we obtain a spectral sequence
\[E_1^{p,q}=\HH^{p+q}(Y,\Gr^W_{-p} B^\bullet) \Rightarrow \HH^{p+q}(Y,B^\bullet)\]
for the filtered complex $(B^\bullet, W^\bullet)$. By \cite[Thm.\,3.18]{PS08}, this sequence degenerates at $E_2$, so we have $E_2^{p,q} = \Gr^W_{-p}\HH^{p+q}(Y,B^{\bullet})$.
 
To compute this spectral sequence, consider the complex $B^{\bullet,\bullet}$ (for $B^{\bullet,\bullet}$ as above), given by
\[\Omega^\bullet_{V(x)} \oplus \Omega^\bullet_{V(y)}\oplus \Omega^\bullet_{V(z)}\stackrel{\delta}{\longrightarrow} \Omega^\bullet_{V(x,y)}\oplus \Omega^\bullet_{V(y,z)}\oplus \Omega^\bullet_{V(x,z)}.\]
Noting that $E_1^{p,q} = \HH^{p+q}(Y,\Gr^W_{-p} B^\bullet) = \HH^q(Y,B^{\bullet,p})$, taking hypercohomology of the columns yields
\begin{center}
\fbox{
\begin{minipage}[c][0.25\textheight]{1\textwidth}
\xymatrix@C=30pt
{
E_1^{p,2}\colon \HH^2(B^{\bullet,0}) \ar^\delta[r]&\HH^2(B^{\bullet,1})\\
E_1^{p,1}\colon \HH^1(B^{\bullet,0}) \ar^\delta[r]&\HH^1(B^{\bullet,1})\\
E_1^{p,0}\colon \HH^0(B^{\bullet,0}) \ar^\delta[r]&\HH^0(B^{\bullet,1})\\
}
\end{minipage}}
\ $\leadsto$\ 
\fbox{
\begin{minipage}[c][0.25\textheight]{1\textwidth}
\xymatrix@C=30pt
{
 \protect\substack{H^2(V(x),\CC)\\ \oplus H^2(V(y),\CC)\\ \oplus H^2(V(z),\CC)} \ar^\delta[r]&0\\
 0&0\\
 \protect\substack{H^0(V(x),\CC)\\ \oplus H^0(V(y),\CC)\\ \oplus H^0(V(z),\CC)}\ar^\delta[r] & \protect\substack{H^0(V(x,y),\CC)\\ \oplus H^0(V(y,z),\CC)\\ \oplus H^0(V(x,z),\CC)}
}
\end{minipage}}
\end{center}
To compute $E_2^{p,q}$, and hence $\Gr^W_{-p}\HH^{p+q}(Y,B^{\bullet})$, we compute the cohomology of these complexes. Noting that the bottom row computes the usual cohomology of a circle $S^1$, we find
\begin{eqnarray*}h^{p,q}H^2(Y,\CC) =& \Gr_F^p\Gr^W_{p+q-2}\HH^{2}(Y,B^{\bullet}) &= \left\{\begin{array}{ll} 3 &\quad p=q=1\\ 0&\quad\hbox{otherwise,}\end{array}\right.\\
h^{p,q}H^1(Y,\CC) =& \Gr_F^p\Gr^W_{p+q-1}\HH^{1}(Y,B^{\bullet}) &= \left\{\begin{array}{ll} 1 &\quad p=q=0\\ 0&\quad\hbox{otherwise,}\end{array}\right.\\
h^{p,q}H^0(Y,\CC)  =&  \Gr_F^p\Gr^W_{p+q}\HH^{0}(Y,B^{\bullet}) &= \left\{\begin{array}{ll} 1 &\quad p=q=0\\ 0&\quad\hbox{otherwise.}\end{array}\right.
\end{eqnarray*}

Note that the Hodge structures on all these cohomology groups are pure, but the one on $H^1(Y,\CC)$ is concentrated in weight zero rather than weight one, so it lies below the cohomology degree which is typical for compact spaces. To obtain an example with mixed weights, one may take a product of $Y$ with an elliptic curve.
\end{example}

Before we can continue our discussion of mixed Hodge structures on a normal crossing space, we need a definition.

\begin{definition} 
The \emph{dual intersection complex} $\Gamma$ of $Y$ is the simplicial complex with
\begin{itemize}
\item vertices indexed by the components $Y_i$ of $Y$, and
\item one $k$-simplex, with vertices $Y_{i_1},...,Y_{i_{k+1}}$, for each connected component of $Y_{i_1}\cap...\cap Y_{i_{k+1}}$.
\end{itemize} 
In a slight abuse of notation, we also use $\Gamma$ to denote the topological space given by the dual intersection complex.
\end{definition}

In Ex.~\ref{example-MHS4}, the dual intersection complex is a necklace of three intervals, glued to give a circle $\Gamma \simeq S^1$. In Sect. \ref{sect:Hodgetoric}, we will discuss another kind of degeneration, called a \emph{toric degeneration}. In such degenerations, the dual intersection complex will be endowed with further structure in addition to being a topological space (and will be denoted by $B$).
 
We saw in Ex.~\ref{example-MHS4} that the cohomology of the dual intersection complex $\Gamma \simeq S^1$ played a role in the calculation of the hypercohomology of $B^{\bullet}$. This is an example of a general phenomenon. Indeed, the lemma below shows that the dual intersection complex $\Gamma$ determines the ranks of the lowest grade pieces of the mixed Hodge structure on $H^n(Y,\CC)$, as follows.

\begin{lemma} 
\label{gradeds-of-Y}
$\Gr_{-p}^W H^p(Y,\CC)=H^p(\Gamma,\CC)$.
\end{lemma}

\begin{proof} 
Consider the spectral sequence
\[E_1^{p,q}=\HH^{p+q}(Y,\Gr^W_{-p} B^\bullet) \Rightarrow \HH^{p+q}(Y,B^\bullet)\]
as constructed in Ex. \ref{example-MHS4}. By \cite[Thm.\,3.18]{PS08}, this sequence degenerates at $E_2$.

Now, recall from Eq. \eqref{B-gradeds} that
\[\Gr_{-p}^WB^{p+q}=B^{q,p}= \bigoplus_i(a_{p})_*\Omega^q_{\tilde{Y}^{(p)}_i},\]
so the complex
\[E_1^{0,0}\stackrel{d_1}{\longrightarrow}  E_1^{1,0}\stackrel{d_1}{\longrightarrow}  E_1^{2,0}\stackrel{d_1}{\longrightarrow} \cdots\]
can be identified with the complex
\[
0\lra\HH^0(B^{\bullet,0})\stackrel{\delta}{\longrightarrow} \HH^0(B^{\bullet,1}) \stackrel{\delta}{\longrightarrow}\HH^0(B^{\bullet,2}) \lra \cdots
\]
which, by our earlier discussion, can be identified with
\[
0\lra 
\bigoplus_{1\le i_0\le N}H^0(Y_{i_0},\CC)\stackrel{\delta}{\longrightarrow} 
\bigoplus_{1\le i_0<i_1\le N}H^0(Y_{i_0}\cap Y_{i_1},\CC)
%\stackrel{\delta}{\longrightarrow} 
%\bigoplus_{1\le i_0<i_1<i_2\le N}H^0(Y_{i_0}\cap Y_{i_1}\cap Y_{i_2},\CC)
\lra \cdots.\]
This in turn can be identified with the \v{C}ech complex of $\Gamma$, so its cohomology is $H^\bullet(\Gamma,\CC)$. Thus, we see that $E_2^{p,0} = H^p(\Gamma,\CC)$.

But the spectral sequence $E_2^{p,q}$ degenerates at $E_2$, so we also have 
\[E_2^{p,0} = \Gr_{-p}\HH^p(Y,B^{\bullet}) = \Gr_{-p}H^p(Y,\CC),\]
as required.\end{proof}

\subsubsection{Cohomological Mixed Hodge Complex of the Limiting Mixed Hodge Structure}

Next we consider the case where our normal crossing divisor arises as the central fibre of a degeneration. Our aim is to give a cohomological mixed Hodge complex $A^\bullet$ that computes the limiting mixed Hodge structure, as defined in Sect. \ref{schmidswork}, as the limit of the polarized variation of Hodge structures on $\Delta\setminus\{0\}$ coming from the variation of cohomology of the (smooth) nearby fibres. Moreover, we will see that there is a close relationship between $A^{\bullet}$ and the cohomological mixed Hodge complex $B^{\bullet}$ introduced in the last subsection. The interested reader may find a more detailed discussion in  \cite[Sect. 4]{St75} and \cite[Sect. 11.2]{PS08}.

Let $\Delta$ be the unit disc and let $b$ be a coordinate on $\Delta$. Assume that we have a projective map $f\colon X\ra \Delta$ onto $\Delta$ that is smooth away from a reduced simple normal crossing central fibre $Y=f^{-1}(0)$. Assume further that $X$ is smooth, so that locally $f$ has the form $f=z_1\ldots  z_k$, for $z_i$ part of a set of local coordinates.

Begin by considering the logarithmic de Rham complex $\Omega^\bullet_X(\log Y)$. In a similar way to Ex. \ref{omegalogDex}, it carries an increasing filtration $W^Y$ by the order of poles,
\begin{align*}
W^Y_k\Omega^{p}_{{X}}(\log Y) = {} &
\Omega^{k}_{{X}}(\log Y)\wedge \Omega^{p-k}_{{X}}.
\end{align*}
Moreover, there is also the Hodge filtration
\[F^k\Omega^{\bullet}_{{X}}(\log Y)=\Omega^{\bullet\ge k}_{{X}}(\log Y).\]

Consider the double complex\footnote{Note that here we use the original notation by Steenbrink \cite{St75}; the two indices $p,q$ are swapped in \cite{PS08}.}
\[A^{p,q}= \Omega^{p+q+1}_{{X}}(\log Y)\,/\,W^Y_q \Omega^{p+q+1}_{{X}}(\log Y),\]
where the first differential is the usual holomorphic de Rham differential and the second differential is given by wedging with $\dlog f = f^*\dlog b$.

Let $A^\bullet$ denote the total complex of $A^{\bullet,\bullet}$. We have three filtrations $W$, $W^Y$ and $F$ on $A^{\bullet}$, given by the rule
\[W_k A^r = \bigoplus_{p+q=r} W^Y_{2q+k+1} A^{p,q} =  \bigoplus_{p+q=r}W^Y_{2q+k+1} \Omega^{p+q+1}_{{X}}(\log Y)\,/\,W^Y_q\Omega^{p+q+1}_{{X}}(\log Y)\]
for $W$, and respectively for $W^Y$ and $F$ in terms of the corresponding filtrations on $\Omega^{p+q+1}_{{X}}(\log Y)\,/\,W^Y_q\Omega^{p+q+1}_{{X}}(\log Y)$, yielding
\[W^Y_kA^r = \bigoplus_{p+q=r} W^Y_{k+q+1} A^{p,q}
\qquad \mathrm{and} \qquad
F^kA^r = \bigoplus_{p+q=r} F^{k+q+1}A^{p,q}.  \]

The injection
\[\dlog f\wedge\colon \Omega^p_{{X}/\Delta}(\log Y)\otimes \calO_Y \lra A^{p,0}\]
turns $A^{\bullet,\bullet}$ into a resolution of $\Omega^\bullet_{ X/\Delta}(\log Y)\otimes_{\shO_X} \calO_Y$, so
\[\HH^n(Y,A^\bullet)= \HH^n(Y,\,\Omega^\bullet_{ X/\Delta}(\log Y)\big|_Y)\]
and one can show that this is equal to the $n$th cohomology of the pullback of $X\setminus Y$ to the universal cover of $\Delta\setminus \{0\}$, which is in turn homotopic to the nearby fibre. In fact, the association
\[\Delta\ni b\longmapsto \HH^n (f^{-1}(b),\,\Omega^\bullet_{ X/\Delta}(\log Y)\big|_{f^{-1}(b)})\]
gives a vector bundle with connection that has log poles on $\Delta$. This vector bundle is the canonical extension (see Sect. \ref{schmidswork}) of the flat bundle defined by the $n$th cohomology variation of Hodge structure over the punctured disc
\[(\Delta\setminus\{0\})\ni b \longmapsto \HH^n(f^{-1}(b),\,\Omega^\bullet_{ (X\setminus Y)/(\Delta\setminus\{0\})}\big|_{f^{-1}(b)}).\]

By \cite[Thm. 4.19]{St75}, $A^\bullet$ is the $\CC$-part of a cohomological mixed Hodge complex that computes the limiting mixed Hodge structure of Thm. \ref{limmhsthm}.
Furthermore, there is an endomorphism of this double complex $\nu\colon A^{p,q}\ra A^{p-1,q+1}$,  given by the natural projection modulo $W^Y_{q+1}$. 
One can show that it is related to the monodromy operator $T$ by
\[\log T=2\pi i \nu,\]
so $\nu$ coincides with $\res(\nabla)$ up to sign; further details may be found in \cite[Thm. 11.21 and Cor. 11.17]{PS08}. 

In this case, one can show that $T$ is actually unipotent; otherwise (e.g. when $Y$ is not reduced) we would have to take its unipotent part as in Sect. \ref{schmidswork}. We find that
\[\ker(\nu)^{\bullet}= W^Y_0 A^{\bullet}\]
is itself a cohomological mixed Hodge complex with the filtrations $W$ and $F$ induced from $A^{\bullet}$, i.e. the injection
\[\operatorname{spe}\colon \ker(\nu)^{\bullet}\lra A^\bullet\]
induces an injection on bigraded pieces. 
In fact, there is an isomorphism
\[\ker(\nu)^{\bullet}\cong B^\bullet,\]
where $B^\bullet$ is the cohomological mixed Hodge complex computing the mixed Hodge structure on the cohomology of $Y$ introduced in Sect. \ref{mhsncd}. This isomorphism is given by taking residues
\[W_0^YA^{p,q}=\frac{W_{q+1}^Y\Omega^{p+q+1}_X(\log Y)}{W_{q}^Y\Omega^{p+q+1}_X(\log Y)}\stackrel{R}{\longrightarrow} B^{p,q} = (a_{q})_*\Omega^p_{\tilde{Y}^{(q)}}.\]

Using this it can be shown \cite[Lemma 4.4]{GKR12} that there is a short exact sequence of cohomological mixed Hodge complexes
\[0\lra B^\bullet\stackrel{\operatorname{spe}}{\longrightarrow} A^\bullet \lra \frac{A^\bullet}{B^\bullet}\lra 0\]
that gives a long exact sequence of mixed Hodge structures on the cohomology sequence
\[\cdots\lra H^i(Y,\CC)\lra H^i(Y,\psi_f\CC)\lra H^{i}(Y,\phi_f\CC)\lra\cdots,\]
where $\psi_f\CC$ and $\phi_f\CC$ are the (perverse) sheaves of nearb$\psi$ and v$\phi$nishing cycles respectively; we refer the interested reader to \cite{Del73b} \cite{Del73a} or \cite[Ch. 13]{PS08} for more information about these sheaves. 

In particular, we note that
\[H^i(Y,\psi_f\CC)\cong H^i(f^{-1}(b),\CC)\] 
for $b\neq 0$, and its mixed Hodge structure is the limiting mixed Hodge structure from Thm. \ref{limmhsthm}. We will call this long exact sequence \emph{Steenbrink's sequence}. 

\begin{example} 
\label{example-MHS5}
Consider the degenerating family given by projecting 
\[X=V(xyz-b)\subset\PP^2\times\Delta \to \Delta.\] 
Note that the central fibre $Y$ of this family is exactly the variety $Y$ from Ex. \ref{example-MHS4}. 

The cohomology of the nearby fibre is the hypercohomology of the logarithmic de Rham complex
\[0\lra\shO_Y\stackrel{d}{\longrightarrow} \left.\Omega_{X/\Delta}(\log Y)\right|_Y\lra 0\]
where $\Omega^2_{X/\Delta}(\log Y)|_Y=0$ because the relative dimension of $f\colon X\ra\Delta$ is one.
To obtain the weight filtration, we use the resolution of this complex by the double complex $A^{\bullet,\bullet}$ as given above. It takes the shape
\begin{center}
\centerline{
\xymatrix@C=30pt
{
0\\
\frac{\Omega^2_{X}(\log Y)}{W^Y_1\Omega^2_{X}(\log Y)}\ar^d[r]\ar^{\wedge \frac{db}b}[u] & 0\\
\frac{\Omega_{X}(\log Y)}{\Omega_{X}}\ar^d[r]\ar^{\wedge \frac{db}b}[u]& \frac{\Omega^2_{X}(\log Y)}{\Omega^2_{X}}\ar[r]\ar[u]&0\\
}}
\end{center}
and by taking residues this becomes
\begin{center}
\centerline{
\xymatrix@C=30pt
{
0\\
\bigoplus_{1\le i<j\le 3} \shO_{Y_i\cap Y_j}\ar^d[r]\ar^{\delta}[u] & 0\\
\bigoplus_{i=1}^3 \shO_{Y_i}\ar^d[r]\ar^{\delta}[u]& \Omega_Y(\log Y) \ar[r]\ar[u]&0\\
}}
\end{center}
where $\Omega_Y(\log Y)$ denotes the subsheaf of $\bigoplus_{i=1}^3 \Omega_{Y_i}(\log \bigcup_{k\neq i} Y_i\cap Y_k)$ given by sections whose residues at the same point add up to zero. A choice of non-trivial global section of $\Omega_Y(\log Y)$ identifies it with $\shO_Y$, for which we already know an acyclic resolution: namely the left column of $A^{\bullet,\bullet}$ 
\[0\lra\shO_Y\ra \bigoplus_{1\le i\le 3} \shO_{Y_i}\stackrel{\delta}{\lra} \bigoplus_{1\le i<j\le 3} \shO_{Y_i\cap Y_j}\lra 0.\]

Now, note that 
\[\CC^3\cong \bigoplus_{1\le i\le 3} \Gamma(Y,\shO_{Y_i})\stackrel{\delta}{\lra} \bigoplus_{1\le i<j\le 3} \Gamma(Y,\shO_{Y_i\cap Y_j})\cong\CC^3\]
computes the usual cohomology of a circle, so we obtain the following Hodge numbers
\begin{eqnarray*}h^{p,q}\HH^2(A^\bullet)&=&\left\{\begin{array}{ll} 1 &\quad p=q=1\\ 0&\quad\hbox{otherwise,}\end{array}\right.\\
h^{p,q}\HH^1(A^\bullet)&=&\left\{\begin{array}{ll} 1 &\quad p=q=0\hbox{ or }p=q=1\\ 0&\quad\hbox{otherwise,}\end{array}\right.\\
h^{p,q}\HH^0(A^\bullet)&=&\left\{\begin{array}{ll} 1 &\quad p=q=0\\ 0&\quad\hbox{otherwise.}\end{array}\right.
\end{eqnarray*}
We see that this has the same ranks as the cohomology of an elliptic curve, which is indeed the nearby fibre. However, the Hodge structure on $\HH^1(A^\bullet)$ is genuinely mixed, unlike the pure Hodge structure of an elliptic curve.

The logarithm of the monodromy $N$ gives an isomorphism
$$\Gr_{1}^W\HH^1(A^\bullet)\lra \Gr_{-1}^W\HH^1(A^\bullet)(-1),$$
so $T\in \End(H^1(Y,\psi_f\CC))$ is given by the matrix
$$\begin{pmatrix} 1 & 1\\ 0 & 1\end{pmatrix}$$
which is the cohomological result of a Dehn twist.

As an enlightening exercise, the reader may wish to give a description of Steenbrink's long exact sequence of mixed Hodge structures in this case, by combining the results of this example with those of Ex. \ref{example-MHS4}.
\end{example}

In this section, we have defined the monodromy operation using the operator $\nu\colon A^{p,q}\ra A^{p+1,q-1}$. In Sect. \ref{sect:Hodgetoric}, however, we will deal with a situation where $A^{\bullet,\bullet}$ is not available, but log differential forms still are. In this case, to define the monodromy operation, one considers the exact sequence
\begin{equation}
\label{relative-log-diff-seq}
0\ra f^*\Omega^1_\Delta(\log \{0\})\otimes_{\shO_X}  \Omega^r_{X/\Delta}(\log Y)\ra \Omega^{r+1}_X(\log Y) \ra \Omega^{r+1}_{X/\Delta}(\log Y)\ra 0
\end{equation}
given by wedging forms. Then one has:
 
\begin{theorem}\textup{\cite[(2.19)-(2.21)]{St75}}
\label{monodromy-as-connecting-homo}
The connecting homomorphism on cohomology
\[\HH^r(Y,\Omega^{\bullet}_{X/\Delta}(\log Y)\big|_Y)\longrightarrow \Omega^1_\Delta(\log \{0\})\otimes\HH^r(Y,\Omega^{\bullet}_{X/\Delta}(\log Y)\big|_Y)\]
obtained from \eqref{relative-log-diff-seq} followed by evaluation on $b\partial_b$ coincides with $\res(\nabla)=\frac{-1}{2\pi i}\log T$.
\end{theorem}

\subsubsection{Maximally Unipotent Degenerations}

Besides Steenbrink's, there are other long exact sequence of mixed Hodge structures involving the monodromy operator $N = \log T$. One such is known as the \emph{Wang sequence}:
\[ \cdots\lra H^k(X\setminus Y)\lra H^k(\psi_f)\lra H^k(\psi_f)(-1) \lra\cdots\]
The most aesthetic way of presenting it is by intertwining it with the long exact sequence of the pair $(X,X\setminus Y)$ and the \emph{Clemens-Schmid exact sequence} to obtain the diagram shown in Fig. \ref{fig:les-nc-deg} (see also \cite[Cor. 11.44]{PS08}).
\begin{figure}[t]
\centering
\vspace{3mm}
%\sidecaption[t]
\includegraphics[width=0.9\textwidth]{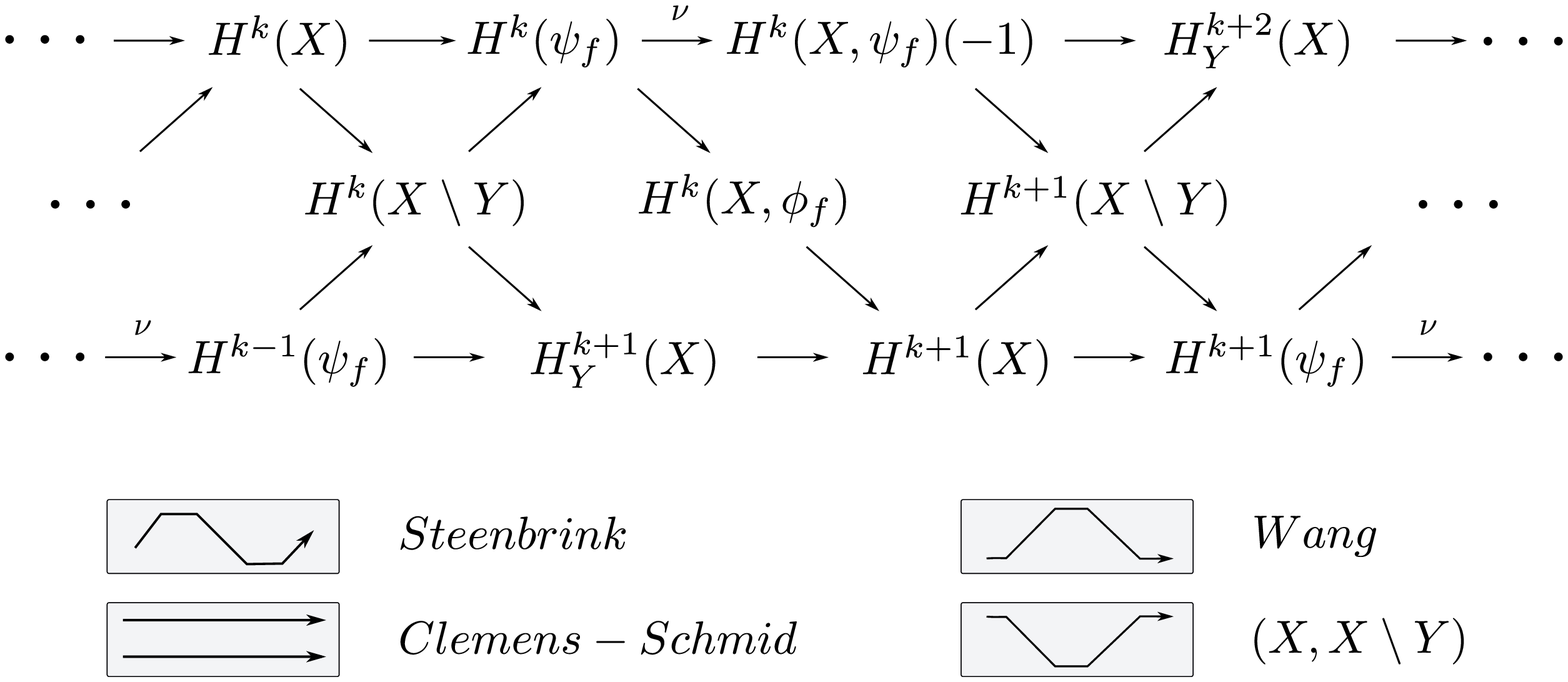}
\caption{Long exact sequences for a normal crossing degeneration}
\label{fig:les-nc-deg}      
\end{figure}
Note that the Clemens-Schmid exact sequence is a 4-term sequence (given by the horizontal lines in Fig.~\ref{fig:les-nc-deg}) and that there is a jump in degrees $k\ra k+2$, rather than $k\ra k+1$. 

As an example, we study this sequence in the setting of Ex.~\ref{example-MHS5}.

\begin{example} 
\label{example-Clemens-Schmid}
In the setting of  Ex.~\ref{example-MHS5}, the central fibre $Y$ is a necklace of three projective lines, so we have
\[\dim H^i(X)=\dim H^i(Y)=\left\{\begin{array}{ll} 3,\ \ &i=2\\ 1, & i=1 \\ 1,&i=0. \end{array}\right.\]
Moreover, the nearby fibre cohomology $H^\bullet(\psi_f)$ is that of an elliptic curve.
In Ex.~\ref{example-MHS5}, we found that $N=\log T$ gives an isomorphism of one-dimensional vector spaces
\[N\colon Gr_1^W H^1(\psi_f)\lra Gr_{-1}^W H^1(\psi_f)\]
and that $N$ is trivial on $H^i(\psi_f)$ for $i\neq 1$. Hence the Clemens-Schmid exact sequence decomposes into the following exact sequences, with ranks indicated by numbers underneath the cohomology groups.

Even degrees:
\[\begin{array}{c}0 \lra \underset{1}{H^{0}(X)}\lra \underset{1}{H^{0}(\psi_f)} \lra 0,\\[5mm]
0 \lra \underset{1}{H^{0}(\psi_f)} \lra \underset{3}{H^{2}_Y(X)}\lra  \underset{3}{H^{2}(X)}\lra \underset{1}{H^{2}(\psi_f)}\lra 0,\\[5mm]
0 \lra \underset{1}{ H^{2}(\psi_f)}\lra \underset{1}{H^{4}_Y(X)}\lra 0\end{array}\]
Odd degrees:
\[\begin{array}{c}0\lra \underset{0}{H^{1}_Y(X)}\lra \underset{1}{H^{1}(X)} \lra \underset{1}{Gr_{-1}^W H^{1}(\psi_f)},\\[5mm]
0\lra \underset{1}{Gr_{1}^W H^{1}(\psi_f)} \stackrel{\nu}{\lra} \underset{1}{Gr_{-1}^W H^{1}(\psi_f)} \lra 0,\\[5mm]
0 \lra \underset{1}{Gr_{1}^W H^{1}(\psi_f)}\lra \underset{1}{H^{3}_Y(X)}\lra  \underset{0}{H^{3}(X)}\lra 0\end{array}\]
\end{example}

In the remainder of this section we are going to present an interesting application of the Clemens-Schmid exact sequence, using it prove a theorem relating the degree of unipotency of $T$ to the cohomology of the dual intersection complex $\Gamma$ of $Y$. This application is taken from Morrison's article \cite{csesa}; we recommend this article for any reader interested in learning more about the Clemens-Schmid exact sequence and its applications. We begin with a definition.

\begin{definition} 
The normal crossing degeneration $f\colon X\ra\Delta$ is called \emph{maximally unipotent} if 
\[(T-\id)^m\neq 0\] 
for $m=\dim Y$ the dimension of a fibre of $f$.
\end{definition}

Recall that if $T$ is unipotent then, by Thm.~\ref{monodromythm}, we have $(T-\id)^{m+1}=0$; this explains the inclusion of the word \emph{maximal} in this definition. Maximally unipotent degenerations are particularly interesting as they play a central role in mirror symmetry \cite{lbhsi} \cite{msrcqtgm} \cite{RS14}. We have the following simple criterion for maximal unipotency.

\begin{theorem} The degeneration $f\colon X\ra\Delta$ is maximally unipotent if and only if
$H^m(\Gamma,\CC)\neq 0$. 
\end{theorem}

\begin{remark} In particular, this implies that a necessary condition for maximal unipotency is that $\Gamma$ has $m$-cells. In other words, $Y$ needs to have zero-dimensional strata, i.e. points where $m+1$ components of $Y$ come together. So a degeneration of a surface into two surfaces meeting along a curve is not maximally unipotent.
\end{remark}

\begin{proof} 
Maximal unipotency is equivalent to $N^{m}\neq 0$ on $H^m(\psi_f)$. 
By the construction of the weight filtration
in Sect. \ref{sec-weights-from-endo}, $N^{m}$ induces an isomorphism
\[N^{m}\colon \Gr^W_{m}H^m(\psi_f) \lra \Gr^W_{-m}H^m(\psi_f).\]
Since $N^{m+1}=0$, we have $\Gr^W_{-m}H^m(\psi_f)=W_{-m}H^m(\psi_f)$, so maximal unipotency is equivalent to $\Gr^W_{-m}H^m(\psi_f)\neq 0$.

Now, the Clemens-Schmid exact sequence
\[\cdots \lra H^{m-2}_Y(X)\lra  H^{m}(X)\lra H^{m}(\psi_f) \stackrel{\nu}{\lra} H^{m}(\psi_f) \lra\cdots\]
gives a sequence 
\[\cdots \lra H^{m-2}_Y(X)\lra  H^{m}(X)\lra \ker(N) \lra 0\]
and, since $N^{m+1}=0$, we have $\Gr^W_{-m}H^m(\psi_f)= \Gr^W_{-m}\ker(N)$. Moreover, from the extraordinary cup product \cite[Cor~6.28]{PS08}, one can deduce that $W_k H^{m-2}_Y(X) = 0$ for $k<0$ and, since $X$ retracts to $Y$, we thus have isomorphisms
\[\Gr^W_k H^{m}(Y)\lra \Gr^W_k\ker(N)\]
for $k<0$. In particular $\Gr_{-m}^W H^m(Y)=\Gr^W_{-m}\ker(N)$, so the assertion follows from Lemma~\ref{gradeds-of-Y}.
\end{proof}

\section{Hodge Theory of Toric Degenerations} \label{sect:Hodgetoric}

In \cite{eagsp}, van Garrel, Overholser and Ruddat give an introduction to the program of Gross-Siebert, the main aim of which is to study mirror symmetry via \emph{toric degenerations}. 
A toric degeneration is in some ways more specific and in other ways more general than a normal crossing degeneration. 
It is more restrictive because it requires the components $Y_i$ of the special fibre $Y$ to be toric varieties. 
On the other hand, it is more general because it allows more general local models for the total space, rather than just $t=z_1\cdot\ldots\cdot z_r$ as in the normal crossing case. 
In particular, the total space $X$ may feature singularities along $Y$. 
The class of degenerations containing both normal crossing degenerations and toric degenerations is the class of \emph{toroidal degenerations} (given by taking the definition of toric degeneration and dropping the condition that the $Y_i$ be toric varieties), these are similar to but yet more general than those defined in \cite{teI}. In the following section we review the Hodge theory of these objects, as presented in \cite{GS10} and \cite{Ru10}.

Whilst a logarithmic de Rham complex can still be defined for a toric degeneration, 
the analogue of the double complex $A^{\bullet,\bullet}$ doesn't work because the total space $X$ lacks smoothness. This means that we don't know how to obtain the filtration $W^Y_i$, as we don't have a suitable replacement for $\Omega_X^i$. Instead, to deal with log-differential forms on singular spaces, one uses logarithmic geometry; see \cite{eagsp}. 

More rigorously, let $f\colon X\ra\Delta$ be a toric degeneration over the unit disc $\Delta$ and set $Y=f^{-1}(0)$. 
By definition, the map $f$ has local models away from the codimension two locus $Z\subset Y$ ($f$ is log smooth on $X\setminus Z$).
Let $j\colon X\setminus Z\hra X$ denote the inclusion. 
Note that when we write $\Omega_X(\log Y)$ and similar objects in the following, we are implicitly referring to $j_*\Omega_X(\log Y)$, as these sheaves are badly behaved along $Z$.

With this definition, the sequence \eqref{relative-log-diff-seq} is still exact in our setting.
Let $(B,\calP)$ denote the dual intersection complex of this toric degeneration, as defined in \cite{eagsp}.
Let $D\subset B$ denote the singular locus of the affine structure of $B$ and $i\colon B\setminus D\hra B$ the inclusion.
The flat integral tangent vectors give a local system $\Lambda$ on $B\setminus D$. 
Similarly, we have integral cotangent vectors $\check\Lambda$ on $B\setminus D$, and $\check\Lambda=\Hom(\Lambda,\ZZ)$.
The pushforward $i_*\bigwedge^r\check\Lambda$ of the exterior powers of these cotangent vectors will be useful in what follows.

Let $\Aff(B,\ZZ)$ denote the sheaf of integral affine maps from $\Lambda$ to $\ZZ$. 
We have an exact sequence
\[0\ra\ZZ\ra \Aff(B,\ZZ)\ra \check\Lambda \ra 0.\]
The extension class $c_B$ of this sequence is known as the \emph{radiance obstruction} of the affine manifold $B\setminus D$; we have
\[c_B\in H^1(B\setminus D,\Lambda)=\Ext^1(\check\Lambda,\ZZ).\]

The associated exact sequences of exterior powers also give exact sequences after pushing forward
\begin{equation}
\label{aff-monodromy-seq}
0\lra\ZZ\otimes_\ZZ i_*\bigwedge^p\check \Lambda \lra i_*\bigwedge^{p+1}\Aff(B,\ZZ)\lra i_*\bigwedge^{p+1}\check\Lambda \lra 0
\end{equation}
and it can be checked that the connecting homomorphism in cohomology
\[ H^{q-1}(B,i_*\bigwedge^{p+1}\check\Lambda)\lra H^q(B,i_*\bigwedge^p\check\Lambda)\]
is given by cup product with $i_*c_B$.

To state the main result of this section, we give loose definitions of a couple of key terms; precise definitions of these terms may be found in \cite{GS03}, \cite{GS10}.

\begin{definition}
\begin{enumerate}
\item \emph{Positivity} for a tropical manifold is a notion of curvature. This notion is very natural, as every $(B,\calP)$ arising as the dual intersection complex of a toric log Calabi-Yau space is positive.
\item Roughly speaking, a tropical manifold $(B,\calP)$ is \emph{supersimple} if the discriminant $\Delta$ is locally a union of tropical hyperplanes.
\end{enumerate}
\end{definition}

The main result in \cite{GS10} is then the following.

\begin{theorem} Let $f\colon X\ra\Delta$ be a toric degeneration such that the dual intersection complex $(B,\calP)$ of $Y=f^{-1}(0)$ is supersimple and positive.
\begin{enumerate}
\item The spectral sequence 
\[E_1^{p,q}=H^q(Y,\Omega^{p}_{X/\Delta}(\log Y)|_Y)\Rightarrow \HH^{p+q}(Y,\Omega^{\bullet}_{X/\Delta}(\log Y)|_Y)\]
degenerates at $E_1$ giving rise to a weak Hodge decomposition
\[\HH^{p+q}(Y,\Omega^{\bullet}_{X/\Delta}(\log Y)|_Y) \cong \bigoplus_{p+q=k}H^q(Y,\Omega^{p}_{X/\Delta}(\log Y)|_Y).\]
\item There are canonical isomorphisms
\begin{eqnarray*}
H^q(Y,\Omega^p_{X/\Delta}(\log Y)|_Y) &=& H^q(B,i_*\bigwedge^p\check\Lambda\otimes\CC) \\
H^q(Y,\shT^p_{X/\Delta}(\log Y)|_Y) &=& H^q(B,i_*\bigwedge^p\Lambda\otimes\CC)
\end{eqnarray*}
where $\shT^p_{X/\Delta}(\log Y)$ denotes the $p$th wedge power of the relative log tangent sheaf.
\item The long exact sequences of cohomology coming from \eqref{relative-log-diff-seq} and \eqref{aff-monodromy-seq} are canonically isomorphic. In particular, the monodromy operator is identified with the wedge product with the radiance obstruction:
\begin{center}
\centerline{
\xymatrix@C=30pt
{
H^q(Y,\Omega^{p}_{X/\Delta}(\log Y)|_Y)\ar^{\res(\nabla)}[r] &
H^{q+1}(Y,\Omega^{p-1}_{X/\Delta}(\log Y)|_Y)\\
H^q(B,i_*\bigwedge^p\check\Lambda\otimes\CC)\ar@{=}[u]\ar^{c_B\wedge}[r]& H^{q+1}(B,i_*\bigwedge^{p-1}\check\Lambda\otimes\CC)\ar@{=}[u]
}}
\end{center}
\item If $\check f\colon \check X\ra\Delta$ is another toric degeneration with dual intersection complex $(\check B,\check\calP)$ that is Legendre dual to $(B,\calP)$ (via some multivalued piecewise linear functions $\varphi,\check\varphi$ respectively) then $(\check B,\check\calP)$ is also supersimple and positive and there are canonical isomorphisms
\begin{eqnarray*}
H^q(B,i_*\bigwedge^p\check\Lambda) &=& H^q(\check B,i_*\bigwedge^p\Lambda),\\
H^q(B,i_*\bigwedge^p\Lambda) &=& H^q(\check B,i_*\bigwedge^p\check\Lambda),
\end{eqnarray*}
that give the mirror duality of Hodge numbers via item 2 above.
\end{enumerate}
\end{theorem}

Dropping the restrictive supersimplicity assumption, parts of this theorem were generalized in \cite{Ru10}. 
This includes the new discovery (related to item 4 in the above theorem) that if one of the mirror partners $X$, $\check{X}$ is an orbifold, then one has a diagram
\begin{equation*} \label{nice}
\begin{array}{ccc}
h_\st^{p,q}(X_\eta) & =\joinrel= & h_\st^{d-p,q}(\check X_\eta)\\
\rotatebox{90}{$\le$} & & \rotatebox{90}{$\le$} \\
h^{p,q}(X_\eta) & & h^{d-p,q}(\check X_\eta)\\
\rotatebox{90}{$\le$} & & \rotatebox{90}{$\le$} \\
h_\aff^{p,q}(Y) & =\joinrel= & h_\aff^{d-p,q}(\check Y)\\
\end{array}
\end{equation*}
where 
\begin{itemize}
\item $d=\dim Y$,
\item $X_\eta$, $\check X_\eta$ denote the generic fibres of $f,\check f$ respectively,
\item $h_\st^{p,q}$ denotes Batyrev's stringy Hodge numbers.
\item $h_\aff^{p,q}(Y):=\rk H^q(B,\bigwedge^p\check\Lambda)$, and similarly $h_\aff^{p,q}(\check Y):=\rk H^q(\check B,\bigwedge^p\check\Lambda)$.
\end{itemize}
Moreover, the differences are mirror dual, i.e.,
\begin{eqnarray*}
h_\st^{p,q}(X_\eta)-h^{p,q}(X_\eta)&=&h^{d-p,q}(\check X_\eta)-h_\aff^{d-p,q}(\check Y),\\
h_\st^{p,q}(\check X_\eta)-h^{p,q}(\check X_\eta)&=& h^{d-p,q}(X_\eta)-h_\aff^{d-p,q}(Y).
\end{eqnarray*}
This has been proven under some assumptions, see \cite{Ru10} for details.

\begin{example}[Tropical elliptic curve]
The family $f\colon X\ra\AA^1$ from Ex. \ref{example-MHS5} is a toric degeneration and $B$ is a circle with three integral points, so $B=\RR/3\ZZ$. Then $\calP$ is the subdivision of $B$ into three unit length intervals.
We have $D =\emptyset$ and $\Lambda\cong\check\Lambda\cong\ZZ$ are constant sheaves. The tropical Hodge diamond
\[\begin{array}{ccc} & H^0(B,\ZZ) &  \\
H^0(B,\check\Lambda) & & H^1(B,\ZZ)  \\
& H^1(B,\check\Lambda) & 
\end{array}\]
is indeed that of an elliptic curve, as each term is isomorphic to $\ZZ$.
\end{example}

\noindent {\small \textbf{Acknowledgements.} A part of these notes was written while the authors were in residence at the Fields Institute Thematic Program on Calabi-Yau Varieties: Arithmetic, Geometry and Physics; we would like to thank the Fields Institute for their support and hospitality.}

\bibliography{books}

\providecommand{\bysame}{\leavevmode\hbox to3em{\hrulefill}\thinspace}
\providecommand{\MR}{\relax\ifhmode\unskip\space\fi MR }
% \MRhref is called by the amsart/book/proc definition of \MR.
\providecommand{\MRhref}[2]{%
  \href{http://www.ams.org/mathscinet-getitem?mr=#1}{#2}
}
\providecommand{\href}[2]{#2}
\begin{thebibliography}{10}

\bibitem{bpv}
W.~P. Barth, K.~Hulek, C.~A.~M. Peters, and A.~van~de Ven, \emph{Compact
  complex surfaces}, second ed., Ergebnisse der Mathematik und ihrer
  Grenzgebiete, 3. Fogle, A Series of Modern Surveys in Mathematics, vol.~4,
  Springer-Verlag, 2004.

\bibitem{pmpd}
J.~Carlson, S.~M{\"u}ller-Stach, and C.~A.~M. Peters, \emph{Period mappings and
  period domains}, Cambridge Studies in Advanced Mathematics, vol.~85,
  Cambridge University Press, 2003.

\bibitem{Cl77}
C.~H. Clemens, \emph{Degeneration of {K}{\"a}hler manifolds}, Duke Math. J.
  \textbf{44} (1977), no.~2, 215--290.

\bibitem{DelTH2}
P.~Deligne, \emph{Th{\'{e}}orie de {H}odge {II}}, Inst. Hautes {\'E}tudes Sci.
  Publ. Math. (1971), no.~40, 5--57.

\bibitem{Del73b}
\bysame, \emph{Comparaison avec la th{\'e}orie transcendante}, Groupes de
  Monodromie en G{\'e}om{\'e}trie Alg{\'e}brique, Lect. Notes in Math., vol.
  340, Springer, 1973, pp.~116--164.

\bibitem{Del73a}
\bysame, \emph{Le formalisme des cycles {\'{e}}vanescents}, Groupes de
  Monodromie en G{\'e}om{\'e}trie Alg{\'e}brique, Lect. Notes in Math., vol.
  340, Springer, 1973, pp.~82--115.

\bibitem{DelTH3}
\bysame, \emph{Th{\'{e}}orie de {H}odge {III}}, Inst. Hautes {\'E}tudes Sci.
  Publ. Math. (1974), no.~44, 5--77.

\bibitem{lbhsi}
\bysame, \emph{Local behavior of {H}odge structures at infinity}, Mirror
  symmetry, {II}, AMS/IP Stud. Adv. Math., vol.~1, Amer. Math. Soc., 1997,
  pp.~683--699.

\bibitem{Gr68a}
P.~Griffiths, \emph{Periods of integrals on algebraic manifolds. {I}.
  {C}onstruction and properties of the modular varieties}, Amer. J. Math.
  \textbf{90} (1968), 568--626.

\bibitem{Gr68b}
\bysame, \emph{Periods of integrals on algebraic manifolds. {II}. {L}ocal study
  of the period mapping}, Amer. J. Math. \textbf{90} (1968), 805--865.

\bibitem{Gr69}
\bysame, \emph{On the periods of certain rational integrals. {I}, {II}}, Ann.
  of Math. (2) \textbf{90} (1969), 460--495, 496--541.

\bibitem{Gr70b}
\bysame, \emph{Periods of integrals on algebraic manifolds. {III}. {S}ome
  global differential-geometric properties of the period mapping}, Inst. Hautes
  {\'{E}}tudes Sci. Publ. Math. (1970), no.~38, 125--180.

\bibitem{Gr70}
\bysame, \emph{Periods of integrals on algebraic manifolds: Summary of main
  results and discussion of open problems}, Bull. Amer. Math. Soc. \textbf{76}
  (1970), 228--296.

\bibitem{Gr13}
M.~Gross, \emph{Mirror symmetry and the {S}trominger-{Y}au-{Z}aslow
  conjecture}, Current Developments in Mathematics 2012, Int. Press,
  Somerville, MA, 2013, pp.~133--191.

\bibitem{GKR12}
M.~Gross, L.~Katzarkov, and H.~Ruddat, \emph{Towards mirror symmetry for
  varieties of general type}, Preprint, February 2012,
  \texttt{arXiv:1202.4042}.

\bibitem{GS03a}
M.~Gross and B.~Siebert, \emph{Affine manifolds, log structures, and mirror
  symmetry}, Turkish J. Math. \textbf{27} (2003), no.~1, 33--60.

\bibitem{GS03}
\bysame, \emph{Mirror symmetry via logarithmic degeneration data {I}}, J.
  Differential Geom. \textbf{72} (2006), no.~2, 169--338.

\bibitem{GS10}
\bysame, \emph{Mirror symmetry via logarithmic degeneration data {II}}, J.
  Algebraic Geom. \textbf{19} (2010), no.~4, 679--780.

\bibitem{GS11}
\bysame, \emph{An invitation to toric degenerations}, {G}eometry of Special
  Holonomy and Related Topics, Surv. Differ. Geom., vol.~16, Int. Press,
  Somerville, MA, 2011, pp.~43--78.

\bibitem{gmk3s}
A.~Harder and A.~Thompson, \emph{The geometry and moduli of {K3} surfaces},
  Preprint, 2015, to appear.

\bibitem{teI}
G.~Kempf, F.~Knudsen, D.~Mumford, and B.~Saint-Donat, \emph{Toroidal embeddings
  {I}}, Lecture Notes in Mathematics, vol. 339, Springer-Verlag, 1973.

\bibitem{Ku98}
V.~Kulikov, \emph{Mixed {H}odge structures and singularities}, Cambridge Tracts
  in Mathematics, vol. 132, Cambridge University Press, 1998.

\bibitem{Ma74}
B.~Malgrange, \emph{Int{\'e}grales asymptotiques et monodromie}, Ann. Sci.
  {\'E}cole Norm. Sup. (4) \textbf{7} (1974), 405--430.

\bibitem{csesa}
D.~Morrison, \emph{The {C}lemens-{S}chmid exact sequence and applications},
  Topics in Transcendental Algebraic Geometry ({P}rinceton, {N}.{J}.,
  1981/1982) (P.~Griffiths, ed.), Ann. of Math. Stud., vol. 106, Princeton
  Univ. Press, 1984, pp.~101--119.

\bibitem{msrcqtgm}
\bysame, \emph{Mirror symmetry and rational curves on quintic threefolds: a
  guide for mathematicians}, J. Amer. Math. Soc. \textbf{6} (1993), no.~1,
  223--247.

\bibitem{PS08}
C.~A.~M. Peters and J.~H.~M. Steenbrink, \emph{Mixed {H}odge structures},
  Ergebnisse der Mathematik und ihrer Grenzgebiete, 3. Fogle, A Series of
  Modern Surveys in Mathematics, vol.~52, Springer, 2008.

\bibitem{Ru10}
H.~Ruddat, \emph{Log {H}odge groups on a toric {C}alabi-{Y}au degeneration},
  Mirror Symmetry and Tropical Geometry, Contemp. Mathematics, vol. 527, AMS,
  Providence, RI, 2010, pp.~113--164.

\bibitem{RSTZ14}
H.~Ruddat, N.~Sibilla, D.~Treumann, and E.~Zaslow, \emph{Skeleta of affine
  hypersurfaces}, Geom. Topol. \textbf{18} (2014), no.~3, 1343--1395.

\bibitem{RS14}
H.~Ruddat and B.~Siebert, \emph{Canonical coordinates in toric degenerations},
  Preprint, September 2014, \texttt{arXiv:1409.4750}.

\bibitem{Sc73}
W.~Schmid, \emph{Variation of {H}odge structure: The singularities of the
  period mapping}, Invent. Math. \textbf{22} (1973), 211--319.

\bibitem{St75}
J.~H.~M. Steenbrink, \emph{Limits of {H}odge structures}, Invent. Math.
  \textbf{31} (1975), no.~3, 229--257.

\bibitem{eagsp}
M.~van Garrel, D.~P. Overholser, and H.~Ruddat, \emph{Enumerative aspects of
  the {G}ross-{S}iebert program}, Preprint, October 2014,
  \texttt{arXiv:1410.4783}.

\bibitem{vG00}
B.~van Geemen, \emph{{K}uga-{S}atake varieties and the {H}odge conjecture}, The
  Arithmetic and Geometry of Algebraic Cycles (Proceedings of the {NATO}
  {A}dvanced {S}tudy {I}nstitute held as part of the 1998 {CRM} {S}ummer
  {S}chool at {B}anff, {AB}, {J}une 7--19, 1998) (B.~B. Gordon, J.~D. Lewis,
  S.~M{\"u}ller-Stach, S.~Saito, and N.~Yui, eds.), NATO Science Series C:
  Mathematical and Physical Sciences, vol. 548, Kluwer Academic Publishers,
  Dordrecht, 2000, pp.~51--82.

\bibitem{Va80}
A.~N. Var{\v{c}}enko, \emph{Asymptotic behaviour of holomorphic forms
  determines a mixed {H}odge structure}, Dokl. Akad. Nauk SSSR \textbf{255}
  (1980), no.~5, 1035--1038.

\bibitem{htcagI}
C.~Voisin, \emph{{H}odge theory and complex algebraic geometry. {I}}, Cambridge
  Studies in Advanced Mathematics, vol.~76, Cambridge University Press, 2007.

\end{thebibliography}
\bibliographystyle{amsplain}

\end{document}